\input ./style/arxiv-general.cfg
\documentclass[aap,MSNbibl,seceqn,dvips]{arximspdf}
\makeatletter
   \@ifpackageloaded{graphicx}{}{\usepackage{graphicx}}
\makeatother
\usepackage{url,breakurl}
\usepackage{mathbh}
\usepackage{stfloats}

\doi{10.1214/14-AAP1075}
\volume{25}
\issue{6}
\pubyear{2015}
\firstpage{3338}
\lastpage{3380}
\docsubty{FLA}

\makeatletter
  \fnbelowfloat
\newcommand{\conv}{\operatorname{conv}}
\newcommand{\rrvert}{\vert}
\newcommand{\rrVert}{\Vert}
\newcommand{\llvert}{\vert}
\newcommand{\llVert}{\Vert}
\newtheorem{teo}{Theorem}[section]
\newtheorem{theorem}[teo]{Theorem}
\newtheorem{cor}[teo]{Corollary}
\newtheorem{lem}[teo]{Lemma}
\newproclaim{remark}[teo]{Remark}
\newproclaim{definition}[teo]{Definition}
\def\P{\mathbb{P}}
\def\E{\mathbb{E}}
\def\cP{\Phi}
\newcommand{\al}{\alpha}
\newcommand{\lam}{\lambda}
\def\mR{\mathbb{R}}
\def\mC{\mathbb{C}}
\def\mZ{\mathbb{Z}}
\def\mN{\mathcal{N}}
\def\tJ{\tilde{J}}
\newcommand{\h}{\tilde{h}}
\newcommand{\tC}{\tilde{C}}
\newcommand{\cL}{{\mathcal L}}
\newcommand{\cX}{{\mathcal X}}
\newcommand{\cY}{{\mathcal Y}}
\newcommand{\cK}{{\mathcal K}}
\def\1{\mathbh{1}}
\def\y{\mathbf{y}}
\def\X{\mathbf{X}}
\makeatother

\begin{document}
\begin{frontmatter}

\title{On the topology of random complexes built over stationary point
processes\thanksref{T1}}
\runtitle{Random complexes}

\begin{aug}
\author[A]{\fnms{D.}~\snm{Yogeshwaran}\corref{}\ead[label=e1]{d.yogesh@isibang.ac.in}}
\and
\author[B]{\fnms{Robert J.}~\snm{Adler}\ead[label=e2]{robert@ee.technion.ac.il}}
\runauthor{D. Yogeshwaran and R.~J. Adler}
\affiliation{Indian Statistical Institute, Bangalore and\\
Technion---Israel Institute of Technology}
\address[A]{Statistics and Mathematics Unit\\
Indian Statistical Institute\\
Bangalore---560059\\
India\\
\printead{e1}}
\address[B]{Faculty of Electrical Engineering\\
Technion---Israel Institute of Technology\\
Haifa 32000\\
Israel\\
\printead{e2}}
\end{aug}
\thankstext{T1}{Supported in part by Israel Science Foundation 853/10,
AFOSR FA8655-11-1-3039 and FP7-ICT-318493-STREP.}

%
\received{\smonth{11} \syear{2012}}
%
\revised{\smonth{9} \syear{2014}}

%
\begin{abstract}
There has been considerable recent interest, primarily motivated by
problems in applied algebraic topology, in the homology of random
simplicial complexes. We consider the scenario in which the vertices of
the simplices are the points of a random point process in $\mathbb
{R}^d$, and the edges and faces are determined according to some
deterministic rule, typically leading to \v{C}ech and Vietoris--Rips
complexes. In particular, we obtain results about homology, as measured
via the growth of Betti numbers, when the vertices are the points of a
general stationary point process. This significantly extends earlier
results in which the points were either i.i.d. observations or the
points of a Poisson process. In dealing with general point processes,
in which the points exhibit dependence such as attraction or repulsion,
we find phenomena quantitatively different from those observed in the
i.i.d. and Poisson cases. From the point of view of topological data
analysis, our results seriously impact considerations of model (non)robustness for statistical inference. Our proofs rely on analysis of
subgraph and component counts of stationary point processes, which are
of independent interest in stochastic geometry.
\end{abstract}

%
\end{frontmatter}

\setcounter{footnote}{1}

\section{Introduction}\label{secintro}

There has been considerable recent interest, primarily motivated by
problems in applied algebraic topology, in the homology of random
simplicial complexes. Two main scenarios have been considered. In the
geometric model, the vertices of the simplices are a random point set,
and the edges and faces are determined according to some deterministic
rule, typically related to the distance between pairs, or general
subsets, of vertices. This has lead, for example, to the \v{C}ech and
Vietoris--Rips complexes on random Euclidean point sets, studied in
papers such as \cite{Kahle11,Kahle13a}, with an extension to the
manifold setting in \cite{Bobrowski13}.

Another approach has been to consider random subgraphs of complete
graphs, leading to a number of papers dealing with the topology of
random complexes generalising Erd\H{o}s--R\'enyi graphs, as in, for
example, \cite
{babson2011fundamental,cohen2011topology,Kahle09,linial2006homological,meshulam2009homological}.
Also, see the recent survey \cite{Kahle13} for progress in this direction.

The current paper is concerned with the first of these approaches,
although with a novel and---from the point of view of both theory and\break 
applications---important change of emphasis. Previous papers on
simplicial complexes built over random point
sets have always assumed that the points were either independent,
identically distributed (i.i.d.) observations from some
underlying distribution on $\mR^d$, or points of a (typically
nonhomogeneous) Poisson point process. Our aim in this paper is to
investigate situations in which the points are chosen from a general
point process, in which the points
exhibit dependence such as attraction or repulsion. From the point of
view of topological data analysis
(TDA) our results, which show that local dependencies can have a major
effect on the growth rates of topological quantifiers such as Betti
numbers, impact on considerations of model (non)robustness for statistical
inference in TDA. We shall not address these
issues here, however, beyond a few comments in Section~\ref{tdasec} below.

To start being a little more specific, given a point process (i.e.,
locally finite random counting measure) $\Phi$ on $\mR^d$,
recall that the random geometric graph $G(\Phi,r)$, for $r > 0$, is
defined as the graph with vertex set $\Phi$ and (undirected) edge set
$\{ (X,Y) \in\Phi^2\dvtx \llVert X-Y\rrVert\leq r\}$. The properties
of random
geometric graphs when $\Phi$ is a Poisson point process or a point
process of i.i.d. points have been analysed in detail (cf. \cite
{Penrose03}), and recently interest has turned to the richer topic of
random simplicial complexes built over these point sets.

A nonempty family $\cK$ of finite subsets of a finite set $V$ (called
\textit{vertices}) is \textit{an abstract simplicial} complex if $\cX\in
\cK
$ and $\cY\subset\cX$ implies that $\cY\in\cK$. Elements of~$\cK
$ are called \textit{faces} or \textit{simplices}, and the \textit
{dimension of
a face} is defined as its cardinality minus 1. We shall be concerned
with two specific complexes (we shall omit the prefix ``abstract
simplicial'' from now on), the \v{C}ech and Vietoris--Rips complexes.
Let $B_x(\varepsilon)$ denote the ball of radius $\varepsilon$
around $x$, and $\cP= \{x_1,x_2,\ldots, x_m\}$ be a finite collection
of points in $\mR^d$.

%
\begin{definition}[(\v{C}ech complexes)]\label{defcechcomplex}
The complex $C(\cP,\varepsilon)$, constructed according to the
following rules, is called the \v{C}ech complex associated to $\cP$
and~$\varepsilon$:
\begin{longlist}[(2)]
\item[(1)] the $0$-simplices of $C(\cP,\varepsilon)$ are the points
in $\cP$;
\item[(2)] an $n$-simplex or $n$-dimensional ``face'' $\sigma
=[x_{i_0},\ldots,x_{i_n}]$ is in $C(\cP,\varepsilon)$ if $\bigcap_{k=0}^{n}
B_{x_{i_k}}(\varepsilon/2) \ne\varnothing$.
\end{longlist}
\end{definition}

%
\begin{definition}[(Vietoris--Rips complexes)]\label{defripscomplex}
The complex $R(\cP,\varepsilon)$, constructed according to the
following rules, is called the Vietoris--Rips complex associated to
$\cP$ and $\varepsilon$:
\begin{longlist}[(2)]
\item[(1)] the $0$-simplices of $R(\cP,\varepsilon)$ are the points
in $\cP$;
\item[(2)] an $n$-simplex or $n$-dimensional ``face'' $\sigma
=[x_{i_0},\ldots,x_{i_n}]$ is in $R(\cP,\varepsilon)$ if
$B_{x_{i_k}}(\varepsilon/2) \cap
B_{x_{i_m}}(\varepsilon/2) \ne\varnothing$ for every $0\le k < m \le n$.
\end{longlist}
\end{definition}

The collection of all faces of dimension at most $k$ is called the
$k$-skeleton of a complex. Observe that the $1$-skeletons of both the
\v{C}ech and Vietoris--Rips complexes are equal and the same as the
geometric graph on $\cP$ with radius $\varepsilon$.
More information on these complexes will be given in Section~\ref
{sectopprelims} when it is needed. Both of these (related) complexes
are important in their own right, with the \v{C}ech complex being of
particular interest since it is known to be homotopy equivalent to the
random Boolean set $\bigcup_{x\in\cP} B_x({\varepsilon})$, which
appears in integral geometry (e.g., \cite{SW}) and continuum
percolation (e.g., \cite{MeesterRoy}). This homotopy equivalence
follows from the nerve theorem \cite{Bjorner95}, Theorem~10.7.
We shall concentrate in this paper on the ranks of the homology
groups---that is, the Betti numbers---of these complexes in the random
scenario. At a heuristic level, the $k$th Betti number $\beta_k$
measures the number of $k$-dimensional cycles or ``holes'' in the
complex. As a consequence of the nerve theorem, $\beta_k = 0$ for $k
\geq d$ for the \v{C}ech complex, and~this is one of the
distinguishing features of the \v{C}ech complex from that of the
Vietoris--Rips complex.

A complementary approach to studying the topological structure of
simplicial complexes is via (nonsmooth) Morse theory, and here results
for Poisson process generated complexes are given in \cite
{Bobrowski11} via results on the Morse theory of the distance function.
Contrasted with this is discrete Morse theory \cite{Forman02}, which
has also been used to study random complexes in \cite
{Kahle09,Kahle11}. In fact, the local structure of Morse critical
points (both nonsmooth and discrete) is often more amenable to
computation than the global structure of the Betti numbers. Thus we
shall also take this route in parts of this paper.

There are some recurring themes and techniques in the analysis of Betti
numbers and Morse critical points, which are intimately related to the
subgraph and component counts of the corresponding random geometric
graph. Thus, from the purely technical side, much of this paper will be
concerned with the intrinsically interesting task of extending the
results of \cite{Penrose03}, Chapter~3, on subgraph and component
counts of Poisson point processes to more general stationary point processes.

Subgraph counts of a random geometric graph are an example of
U-statistics of point processes. Hence, apart from their applications
in this article, our techniques to study subgraph counts of random
geometric graph over general stationary point processes could be useful
to derive asymptotics for many other translation and scale invariant
U-statistics of point processes (e.g., the number of $k$-simplices in a
\v{C}ech or Vietoris--Rips complexes). Also, the results on subgraph
counts are used to derive results about clique numbers, maximum degree
and chromatic number of the random geometric graph on Poisson or
i.i.d. point process (\cite{Penrose03}, Chapter~6) and with a similar
approach, our results can be used to derive asymptotics for clique
numbers, maximum degree and chromatic number of random geometric graphs
over general stationary point processes; see \cite{Sollerhaus12}, Section~4.3.1.

Analysis of subgraph counts will take up all of Section~\ref
{secsubgraphrgg}, the longest section of the paper. From these
results, we shall be able to extract results about Betti numbers (via
combinatorial topology) as well as the numbers of Morse critical
points. In formulating results, we shall relate the topological
features of the random simplicial complexes to known, inherent
properties of the underlying point processes, including joint
densities, void probabilities or Palm void probabilities. The first two
of these properties, along with association properties, are known to be
useful in studying measures of clustering, and their impact on
percolation of random geometric graphs was studied in \cite
{Blaszczy13}. Since our asymptotic results help quantify the impact of
clustering measures such as sub-Poisson and negative association on
topological features of point processes, they provide additional
applications of the tools of B{\l}aszczyszyn and Yogeshwaran \cite
{Blaszczy14} as measures of clustering.

A sampler of some of our main results follows a little necessary notation.

\subsection{Some notation}

We use $\llvert\cdot\rrvert$ to denote Lebesgue measure and
$\llVert\cdot\rrVert$ for the Euclidean
norm on $\mR^d$. Depending on context, $\llvert\cdot\rrvert$ will
also denote the
cardinality of a set. As above, we denote the ball of radius $r$
centred at $x \in\mR^d$ by $B_x(r)$. For $\mathbf{x}= (x_1,\ldots,x_k)
\in
\mR^{dk}$, let $B_{\mathbf{x}}(r) = \bigcup_{i=1}^kB_{x_i}(r)$,
$h(\mathbf{x}) = h(x_1,\ldots,x_k)$ for $h\dvtx \mR
^{dk} \to\mR$ and $ d\mathbf{x}= dx_1 \cdots dx_k$. Let
$\mathbf{1} = (1,\ldots,1)$. We also use the standard Bachman--Landau
notation for asymptotics\footnote{That is, for sequences $a_n$ and
$b_n$ of positive numbers, we write
\begin{eqnarray*}
a_n&=&o(b_n)\quad\iff\quad \mbox{for any $c>0$, there is a
$n_0$ such that $a_n < cb_n$ for all
$n>n_0$};
\\
a_n&=&O(b_n) \quad\iff\quad\mbox{there exists a $c>0$ and a
$n_0$ such that $a_n < cb_n$ for all
$n>n_0$};
\\
a_n&=&\omega(b_n) \quad\iff\quad \mbox{for any $c>0$, there is a
$n_0$ such that $a_n >cb_n$ for all
$n>n_0$};
\\
a_n&=&\Omega(b_n) \quad\iff\quad \mbox{there exists a $c>0$ and a
$n_0$ such that $a_n > cb_n$ for all
$n>n_0$};
\\
a_n&=&\Theta(b_n) \quad\iff\quad a_n=O(b_n)
\mbox{ and } a_n=\Omega(b_n). 
\end{eqnarray*}
}
and say that a sequence of events $A_n, n \geq1$ occurs \emph{with
high probability} (\emph{w.h.p.}) if $\mathbb{P} \{ A_n \} \to
1$ as $n \to\infty$.

\subsection{A result sampler} We shall now describe, without
(sometimes important) precise technical conditions, some of our main
results. Full details are given in the main body of the paper. We start
with $\Phi$, a unit intensity, stationary point process on $\mR^d$,
and set\footnote{Note that our basic setup is a little different
from that of all the earlier papers mentioned above. To compare our
results with existing ones on Poisson or i.i.d. point processes, note
that $r_n^d$ in our results
typically corresponds to $nr_n^d$ elsewhere. For a general
(non-Poisson) point process, (\ref{Phinequn}) provides a more natural
setting.}
%
\begin{equation}
\label{Phinequn} \Phi_n = \Phi\cap\biggl[\frac{-n^{1/d}}{2},
\frac
{n^{1/d}}{2} \biggr]^d.
\end{equation}
Let
\begin{eqnarray*}
&&\beta_k \bigl(C(\Phi_n,r) \bigr),\qquad
\beta_k \bigl(R(\Phi_n,r) \bigr),
\end{eqnarray*}
respectively, denote the $k$th Betti numbers of the \v{C}ech and
Vietoris--Rips complexes based on $\Phi_n$. Note that the $\beta_k$
of a complex depends on the $(k+1)$ skeleton of the complex alone, and
since the $1$-skeletons are the same for both \v{C}ech and
Vietoris--Rips complexes, we have that $\beta_0(C(\Phi_n,r)) = \beta
_0(R(\Phi_n,r))$.

In addition, let $\mathcal{M}_k(\Phi_n)$ denote the set of Morse
critical points (to be defined in Section~\ref{secmorsetheory}) of
index $k \in\{0,\ldots,d\}$ for the distance function
\begin{eqnarray*}
d_n(x) &=& \min_{X \in\Phi_n}\llVert x-X\rrVert,
\end{eqnarray*}
and set
\begin{eqnarray*}
N_k(\Phi_n,r) &=& \bigl\llvert\bigl\{c \in
\mathcal{M}_k(\Phi_n)\dvtx d_n(c) \leq r \bigr\}
\bigr\rrvert.
\end{eqnarray*}

The importance of the critical points stems from the Morse
inequalities, which imply, in particular, that every index $k$ critical
point contributing to $N_k(\Phi_n,r)$ either increases $\beta
_k(C(\Phi_n,r))$ by $1$ or decreases $\beta_{k-1}(C(\Phi_n,r))$ by
$1$. In particular, this implies that $\beta_k(C(\Phi_n,r)) \leq
N_k(\Phi_n,r)$.

This paper is concerned with the behavior, as $n\to\infty$, of $\beta
_k (C(\Phi_n,r_n) )$,
$\beta_k (R(\Phi_n,r_n) )$, $N_k(\Phi_n,r_n) $ and
$\chi(C(\Phi_n,r_n))$, where $\chi$ denotes the Euler
characteristic. In particular, we shall provide closed form expressions
for the asymptotic, normalized first moments of these variables, along
with bounds for second moments for most of them.

Throughout the remainder of this subsection we shall assume that $\Phi
$ is stationary, unit mean and negatively associated (defined
rigorously in Section~\ref{paraAssociatedpp}). Additional side
conditions may also need to hold, but we shall not state them here. Two
simple examples for which everything works are provided by the Ginibre
point process and the simple perturbed lattice. Many of the results
hold for various other sub-classes of point processes as well, but
our nonspecific blanket assumptions allow for ease of exposition. We
divide the results into three classes, depending on the behavior
of $r_n$.
\begin{longlist}
\item[I. \textsc{Sparse regime}: $r_n \to0$.]
Note that since the points of
$\Phi$ only generate edges and faces of the \v{C}ech and
Vietoris--Rips complexes $C(\Phi_n,r)$ and $R(\Phi_n,r)$ when they
are distance less than $r$ apart, and since $\Phi$ has, on, average,
only one point per unit cube, if $r$ is small we expect that both of
these complexes will be made up primarily of the isolated points of
$\Phi$. We describe this fact by calling this the ``sparse''
regime.

Since the $\beta_0$'s are equal for the two complexes, in this setting,
\begin{eqnarray*}
\mathbb{E} \bigl\{\beta_0\bigl(C(\Phi_n,r_n)
\bigr) \bigr\} &=& \mathbb{E} \bigl\{\beta_0\bigl(R(
\Phi_n,r_n)\bigr) \bigr\} = \Theta(n),
\end{eqnarray*}
and for $k \geq1$, there exist functions $f^k \equiv1$ [i.e., $f^k(r)
= 1$, $\forall r$] or $f^k(r) \to 0$, as $r\to0$, depending on the
precise distribution of $\Phi$ and
on the index $k$, such that
%
\begin{eqnarray}
\label{earlyequn} \qquad\mathbb{E} \bigl\{\beta_k\bigl(C(\Phi_n,r_n)
\bigr) \bigr\} &=& \Theta\bigl(n r_n^{d(k+1)}
f^{k+2}(r_n)\bigr),\qquad k \in\{0,\ldots,d-1\},
\nonumber
\\
\mathbb{E} \bigl\{\beta_k\bigl(R(\Phi_n,r_n)
\bigr) \bigr\} &=& \Theta\bigl(n r_n^{d(2k+1)}
f^{2k+2}(r_n)\bigr),\qquad k \geq1,
\\
\nonumber
\mathbb{E} \bigl\{N_k(\Phi_n,r_n)
\bigr\} &=& \Theta\bigl(n r_n^{dk} f^{k+1}(r_n)
\bigr),\qquad k \in\{0,\ldots,d-1\},
\end{eqnarray}
and $\mathsf{Var} (N_k(\Phi_n,r_n) ) =O(\mathbb{E}
\{N_k(\Phi_n,r_n) \} )$, where $\mathsf{Var}
(X ) $ is the variance of $X$.
In addition, $\mathbb{E} \{n^{-1}\chi(C(\Phi_n,r_n)) \}
\to1$.

In the classical Poisson case, studied in the references given above,
it is known that the same results hold with $f^k \equiv1$.

Using stochastic ordering techniques, we shall also show that
clustering of point processes increases the functions $f^k(r)$ and
consequently the mean of the $\beta_k$ and $N_k$ as well. Also, we
know that for the Ginibre point process and for the zeroes of Gaussian
entire functions, $f^k(r) = r^{k(k-1)}$. Thus there is a systematic
difference between the scaling limits for Poisson and at least some
negatively associated point processes.
\end{longlist}

\begin{longlist}
\item[II. \textsc{Thermodynamic regime}: $r_n^d \to\beta\in(0,\infty)$.] In
this regime an edge between two points in $\Phi$, which are,
in a rough sense, an average distance of one unit apart, will be formed
if they manage to get within a distance $\beta^{1/d}$ of one another.
Since, in most scenarios, there should be a reasonable probability of
this happening, we expect to see quite a few edges and, in fact,
simplices and homologies up to dimension $d-1$. Indeed, this is the
case, and the main result in this regime is that topological complexity
grows at a rate proportional to the number of points, in the sense that
\begin{eqnarray*}
\mathbb{E} \bigl\{\beta_k\bigl(C(\Phi_n,r_n)
\bigr) \bigr\} &=& \Theta(n),\qquad k \in\{0,\ldots,d-1\},
\end{eqnarray*}
with identical results for $\mathbb{E} \{\beta_k(R(\Phi
_n,r_n)) \} $ and $ \mathbb{E} \{N_k(\Phi_n,r_n)
\} $ for the appropriate~$k$.
In addition,
$\mathsf{Var} (N_k(\Phi_n,r_n) ) = O(\mathbb{E}
\{N_k(\Phi_n,r_n) \} )$
and
\begin{eqnarray*}
\mathbb{E} \bigl\{n^{-1}{\chi\bigl(C(\Phi_n,r_n)
\bigr)} \bigr\} &\to& 1 + \sum_{k = 1}^d(-1)^k
\nu_k(\Phi,\beta),
\end{eqnarray*}
where the $\nu_k(\Phi,\beta)$ are defined in Theorem~\ref{propmorseptsth}. Since there is no appearance in these results of
an analogue to the $f$ of (\ref{earlyequn}), the normalizations here
have the same orders as in the Poisson and i.i.d. cases.
\end{longlist}

\begin{longlist}
\item[III. \textsc{Connectivity regime}: $r_n^d = \Theta(\log n)$.] Clearly, if
$r_n$ is large enough, there comes a point (which we call the \emph{contractibility radius}) beyond which each point of $\Phi_n$ will be
connected to the others, and the \v{C}ech complex will become
contractible to a single point, while the Vietoris--Rips complex will
become topologically $k$-connected. (This is certainly the case if
$r_n=\sqrt{d}n^{1/d}$.) The question then is ``how large is large enough?''

It turns out that in the current scenario of negative association there
exist case dependent constants $C$ such that for $r_n \geq C(\log
n)^{{1}/{d}}$, $C(\Phi_n,r_n)$ is contractible w.h.p. as $n \to
\infty$. In the specific cases of the Ginibre process or zeroes of
Gaussian entire functions, this happens earlier, and $r_n = \Theta
((\log n)^{{1}/{4}})$ is the radius for contractibility of the \v
{C}ech complex. As a trivial corollary, it follows that, w.h.p. $\chi
(C(\Phi_n,r_n)) = 1$ when $r_n$ is the radius of contractibility.
Further, for the Ginibre process, $r_n = \Theta((\log n)^{{1}/{4}})$
is also the critical radius for $k$-connectedness of the Vietoris--Rips complex.
\end{longlist}

\subsection{Some implications for topological data analysis}\label{tdasec}

Perhaps the core tool of TDA is persistent homology, as visualized
through barcodes and persistence diagrams; cf. \cite
{Carlsson09,edelsharer,Ghrist08,zomorodianbook}. See also the very
accessible recent survey \cite{Carlsson14}. While here is not the
place to go into the details of persistent homology, it can be
described reasonably simply in the setting of this paper. For a given
$n$, and a collection of points $\Phi_n$, consider the collections of
\v{C}ech (or Vietoris--Rips) complexes $C(\Phi_n,r)$ built over
these points, as $r$ grows. Initially, $C(\Phi_n,0)$ will contain only
the points of $\Phi_n$. However, as $r$ increases, different
homological entities (cycles of differing degree) will appear and,
eventually, disappear. If to each such phenomenon we assign an interval
starting at the birth time and ending at the death time, then the
collection of all of these intervals is a representation of the
persistent homology generated by $\Phi_n$ and is known as its {\it
barcode}. The individual intervals are referred to as {\it bars}. The
Betti numbers $\beta_k(C(\Phi_n,r))$ therefore count the number of
bars related to $k$-cycles active at ``connection distance'' $r$.

%
\begin{figure}

\includegraphics{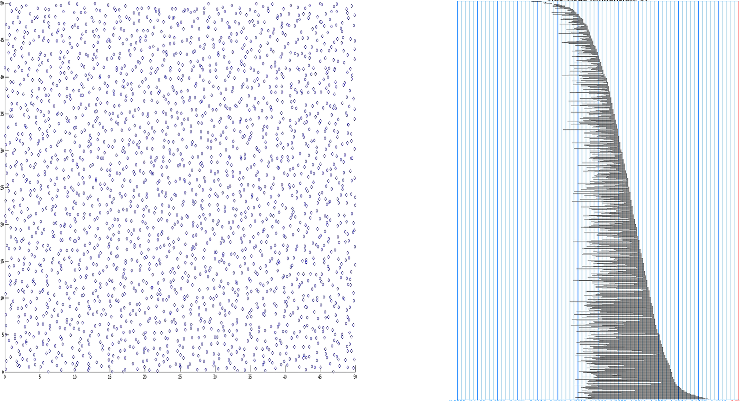}

\caption{Hypergeometric perturbed lattice. $H_1$ barcodes of the Rips complex.}\label{figB1}
\end{figure}

\begin{figure}[b]

\includegraphics{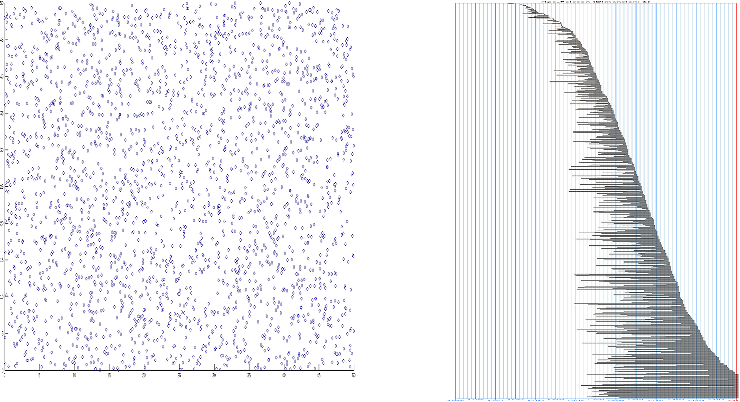}

\caption{Poisson point process. $H_1$ barcodes of the Rips complex.}\label{figB2}
\end{figure}

\begin{figure}

\includegraphics{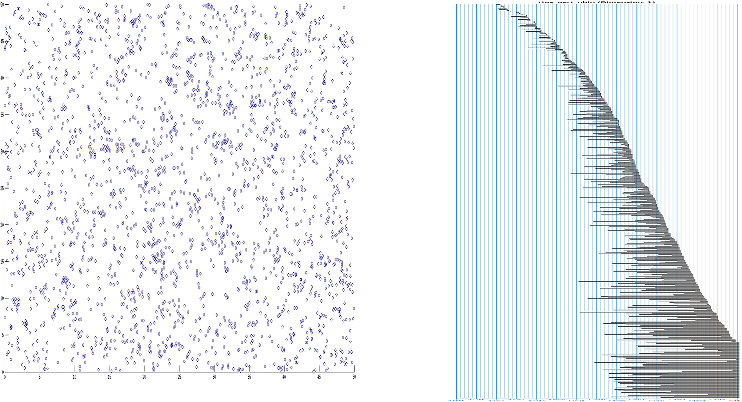}

\caption{Negative Binomial perturbed lattice. $H_1$ barcodes of the Rips complex.}\label{figB3}
\end{figure}

Heuristics and simulations\footnote{These barcodes were simulated
using the easy-to-use and open access package javaPlex \cite
{Javaplex11}.} (see Figures~\ref{figB1},~\ref{figB2},~\ref{figB3})
indicate that as the points are more regularly distributed in a point
process, the bars start later and vanish earlier than those for Poisson
point process. In the three figures, all point processes are of unit
intensity, and we observe that the hypergeometric perturbed lattice has
more regularly distributed points than the Poisson point process, which
in turn has more regularly distributed points than the negative
binomial perturbed lattice. We can see that the corresponding bars
start earlier and end later as we go from hypergeometric perturbed
lattice to the Poisson point process to the negative binomial perturbed
lattice. Some of our results confirm this heuristic. For example, using
the results above it is easy to see that nontrivial homology groups\vspace*{1pt} of
\v{C}ech and Vietoris--Rips complexes start to appear once $r_n$
satisfies $r_n^{d(k+1)}f^{k+2}(r_n) = \omega(n^{-1})$. For the Poisson
case this requires only
$r_n=\omega(n^{-1/d(k+1)})$. Since, typically, $f(r)\to0$ as $r\to
0$, we therefore generally need larger radii for nontrivial homology
(and hence for bars) to appear. The disappearance of homology is
harder, however, and in general, our results on connectivity cannot
confirm the heuristic. However, for the Ginibre point process and
zeroes of GEF in $\mR^2$, they do show that nontrivial topology
vanishes at $r_n = \omega((\log n)^{1/4})$ as opposed to $\omega
((\log n)^{1/2})$ for a two-dimensional Poisson point process of
constant intensity.

As for implications to TDA, applied topologists are beginning to
appreciate the fact that stochasticity underlies their data as a
consequence of sampling, and are beginning to build statistical models
to allow parameter estimation and inference
(e.g., \cite{Bubenik2010,Chazal2011,Turner,Mileyko2011,Bobrowski14}). The results
of this paper show that small changes in model structure (such as the
introduction of attraction and repulsion between points in a data
cloud) can have measurable effects on topological behavior.

The remainder of the paper is organized as follows: in the following
section, we shall summarize some facts needed from the theory of point
processes. Sections~\ref{secsubgraphrgg},~\ref{secbettirgc} and
\ref{secmorsergc} contain the
core technical results on component and subgraph counts, Betti numbers
and Morse critical points, respectively. We shall give careful proofs
for all the results of Section~\ref{secsubgraphrgg} barring the
results on extension to subcomplex counts in Section~\ref
{secsubcomplex} since these mimic earlier proofs. The results of the
Sections~\ref{secbettirgc} and~\ref{secmorsergc} are either easy
corollaries of earlier results or can be proved by using similar
techniques, and so there we shall give less detail. \hyperref[secAppendix]{Appendix} contains a technical
result regarding Palm void probabilities of the Ginibre process which
Manjunath Krishnapur proved for us.

\section{Point processes}\label{secprelims}

Our aim in this section is to set up some general definitions related
to point processes, give some background on those of main interest to
us and to prove two technical results, of some independent interest,
which we shall need later.

\subsection{Point processes and Palm measures}
A point process $\Phi$ in $\mR^d$ is a \mbox{$\mN$-}valued random variable,
where $\mN$ is the space of locally finite (Radon) counting measures
in $\mR^d$ equipped with the canonical $\sigma$-algebra; 
cf. \cite{Kallenberg83,Stoyan95,SW}. We can represent $\Phi$ as
either a random measure, $\Phi(\cdot) = \sum_i \delta_{X_i}(\cdot
)$ or as a random point set $\Phi= \{X_i\}_{i \geq1}$, where, in both
cases, the $X_i$ are the ``points'' of the process.

The factorial moment measure $\alpha^{(k)}$ of a point process $\Phi$
is defined by
\begin{eqnarray*}
\alpha^{(k)} \Biggl(\prod_{i=1}^nB_i
\Biggr) &=& \mathbb{E} \Biggl\{\prod_{i=1}^n
\Phi(B_i) \Biggr\},
\end{eqnarray*}
for disjoint bounded Borel subsets $B_1,\ldots,B_n$. When $k=1$, $\al:=
\alpha^{(1)}$ is called the \emph{intensity} or \emph{mean measure},
and $\alpha^{(k)}$ also serves as the intensity measure of the point process
\begin{eqnarray*}
\Phi^{(k)}:= \bigl\{ (X_1,\ldots,X_k) \in
\Phi^k\dvtx X_i \neq X_j, \forall i \neq j \bigr
\}.
\end{eqnarray*}
The \emph{$k$th joint intensity}, $\rho^{(k)}\dvtx(\mR^{d})^{k} \to
[0,\infty)$ is the density (if it exists) of $\alpha^{k}$ with
respect to
(in this paper) Lebesgue measure. The $\rho^{(k)}$ characterize a
simple point process just as moments characterize a random variable. A
sufficient condition for joint intensities (when they exist) to
characterize a simple point process is $\rho^{(k)}(\cdot) \leq C^k$ for
some constant $C$ and for all $k \geq1$; cf. \cite{Ben09}, Lemma~4.2.6
and Remark 1.2.4. Throughout, we shall restrict ourselves to simple
stationary point processes of unit intensity; namely, $\alpha(B) =
\llvert B\rrvert$ for all bounded, Borel $B$. We also shall assume
that all the
joint intensities $\rho^{(k)}(\cdot)$ exist for the point processes under
consideration in this article.

For a point process $\Phi$ whose probability distribution is $\mathbb
P$, its reduced \emph{Palm probability distribution} $\P^!_x$ at $x \in
\mR^d$ is defined as the probability measure that satisfies the
following disintegration formula for any bounded measurable function
$u\dvtx \mN\times\mR^d \to\mR_+$ with compact support in the second
co-ordinate:
\[
\int_{\mN}\P(d\phi) \int_{\mR^d}\phi(dx)u(\phi,x) = \int_{\mR^d}dx \int_{\mN}
\P^!_x(d\phi)u\bigl(\phi\cup\{x\},x\bigr).
\]

As a consequence of the above definition, for the corresponding Palm
expectation $\E^!_x$ with the function $u$ satisfying assumptions as
above, we get the well-known \emph{refined Campbell theorem} (cf. \cite
{Stoyan95}, page 119, \cite{SW}, Theorem~3.3),
%
%
\begin{equation}
\label{eqnCampbell-Mecke} \mathbb{E} \biggl\{\sum_{X \in\Phi}u(
\Phi,X) \biggr\} = \int_{x \in\mR^d}\E^!_x\bigl\{u\bigl(
\Phi\cup\{x\},x\bigr)\bigr\} \,dx.
\end{equation}

If the point process is not stationary or has unit intensity, one can
still define Palm probability distribution by replacing $dx$ on the
RHS of the above two equations with the intensity measure of the point
process. In particular, the definition of Palm probability gives us
that $\P^!_x\{\Phi(x) = 0\} = 1$. Intuitively, $\P^!_x$ is the
distribution of the remainder of the point process, conditioned on
there having been a point at~$x$. 

\subsection{Some special cases}\label{secclassespp}
We shall assume the reader is familiar with stationary Poisson point
processes, determined, for example, by $\rho^k\equiv1$ for all $k$,
and use this as a basis for comparison in a quick tour through some
non-Poisson cases that will provide examples for the theorems of the
remaining sections.

\subsubsection*{Associated point processes}\label{paraAssociatedpp}

A point process $\Phi$ is called \emph{associated} (or \emph{positively
associated}) if for any finite collection of disjoint bounded Borel
subsets $B_1,\ldots,B_k\subset\mR^d$ and
$f,g$ continuous and increasing functions taking values in~$[0,1]$,
%
\begin{eqnarray}
\label{posassequn} \mathsf{Cov} \bigl( f\bigl(\Phi(B_1),\ldots,
\Phi(B_k)\bigr), g\bigl(\Phi(B_1),\ldots,
\Phi(B_k)\bigr) \bigr) \geq0;
\end{eqnarray}
cf. \cite{BurtonWaymire1985}. The referenced article gives many
examples of associated processes. We call a point process $\Phi$ \emph{negatively associated} if
%
\begin{eqnarray}
\label{negassequn} \mathsf{Cov} \bigl( f\bigl(\Phi(B_1),\ldots,
\Phi(B_k)\bigr), g\bigl(\Phi(B_{k+1}),\ldots,
\Phi(B_l)\bigr) \bigr) \le0,
\end{eqnarray}
for any finite collection of bounded Borel subsets $B_1,\ldots
,B_l\subset\mR^d$ such that $(B_1\cup\cdots\cup B_k)\cap
(B_{k+1}\cup\cdots\cup B_{l})=\varnothing$ and $f,g$ increasing
bounded continuous functions.

In general, the literature contains fewer examples of negatively
associated processes than their positively associated counterparts, a
phenomenon that occurs even in simpler situations; cf. \cite
{Pemantle00}. We shall give two examples of negatively associated point
processes below (determinantal point processes and the simple perturbed
lattice) as these are of more interest to us in this article, but we
refer the reader to \cite{BurtonWaymire1985,Sollerhaus12} for many
examples of positively associated point processes. The stationary
Poisson point process is both negatively and positively associated.
Finite independent unions of negatively associated point processes is
negatively associated as well, and this can be used to construct many
examples of negatively associated point processes from a few simple
examples. Just to reiterate the earlier point about scarcity of
negatively associated point processes, not many ``natural'' examples,
apart from the two presented below and binomial point process, are
known. This is in contrast to the situation for positively associated
point processes. Here are three other point processes of interest to us:

\subsubsection*{Determinantal processes}
A simple point process $\Phi$ on $\mR^d$ is said to be determinantal
with kernel $K\dvtx (\mR^d)^2 \to\mathbb{C}$ if its joint intensities
satisfy the following equality for all $k \geq1$ and for all
$x_1,\ldots,x_k \in\mR^d$:
%
\begin{eqnarray}
\label{equndeterminantal} \rho^{k}(x_1,\ldots,x_k) =
\operatorname{det}\bigl(K(x_i,x_j)_{1 \leq i,j
\leq k}\bigr),
\end{eqnarray}
where $\operatorname{det}$ indicates a determinant of a matrix.

Stationary determinantal point processes with continuous kernels are
negatively associated \cite{Subhro12}, Corollary 6.3. For examples of
stationary determinantal point processes, see
\cite{Lavancier12}, Section~5. A determinantal process of
particular interest is the unit intensity \emph{Ginibre process} (\cite
{Ben09}, Section~4.3.7), which has the continuous kernel
\[
K(z,w) = \exp\bigl(-\tfrac{1}{2}\bigl(\llVert z\rrVert^2 +
\llVert w\rrVert^2\bigr) + z\overline{w}\bigr), \qquad z,w \in\mC.
\]
In \cite{Miyoshi13}, the authors have introduced a family of
determinantal point processes called the $\alpha$-Ginibre point
processes in the context of modeling cellular networks with $\alpha=
1$ corresponding to the Ginibre point process, and as $\alpha\to0$,
the \mbox{$\alpha$-}Ginibre point processes converges to the appropriate
Poisson point process. This class of point processes gives a continuous
family of point processes between the Poisson point process and the
Ginibre point process.

A counterpart to determinantal point processes are permanental point
processes, which can be defined by replacing the determinant in (\ref
{equndeterminantal}) by a matrix permanent.

\subsubsection*{Perturbed lattices}
Let ${N_z\dvtx z \in\mZ^d}$ be i.i.d. integer valued random variables
distributed as $N$, and ${X_{iz}, i \geq1, z \in\mZ^d}$ be
i.i.d. $\mR^d$ valued random variables distributed as $X$. A perturbed
lattice is defined as
\[
\Phi(N,X):= \bigcup_{z \in\mZ^d} \bigcup
_{i=1}^{N_z} \{z + X_{iz}\}, %
\]
provided that $\Phi(N,X)$ is a simple point process. $N$ is called the
{\it replication kernel}, and $X$ is called the {\it perturbation
kernel}. Though the point process is stationary with respect to lattice
translations only, we can make it stationary with respect to $\mR^d$
translations by shifting the origin uniformly within $[0,1]^d$; that
is, $ \bigcup_{z \in\mZ^d} \bigcup_{i=1}^{N_z} \{U + z + X_{iz}\}
$ is stationary if $U$ is uniformly distributed in $[0,1]^d$. The point
process for which $N \equiv1$ and $X$ is uniform on the unit cube is
known as the \emph{simple perturbed lattice} and is negatively
associated. For more details, see below and especially \cite{Blaszczy14}.

\subsubsection*{Zeroes of a Gaussian entire function}
(Normalized) Gaussian entire functions are defined on the complex plane
$\mC$ via the a.s. convergent expansion $f(z) = \sum_{n = 0}^{\infty
}\xi_nz^n/\sqrt{{n!}}$, where the $\xi_n$ are i.i.d. standard
complex Gaussians. The zeros of $f$ (when considered as a point process
in $\mR^2$ and called as zeros of GEF), while neither negatively
associated nor determinantal, share many properties with the Ginibre
process that make them interesting and tractable; cf. \cite{Ben09}
for more background.

\subsubsection*{Sub- and super-Poisson processes}
At times, weaker notions than association, based only on factorial
moment measures, suffice to establish interesting results.

We say that a point process $\Phi_1$ is \emph{$\alpha$-weaker} than
$\Phi_2$ (written $\Phi_1 \leq_{\al-w} \Phi_2$)
if $\al^{(k)}_1(B) \leq\al^{(k)}_2(B)$ for all $k \geq1$ and
bounded Borel $B\subset(\mR^d)^k$.
We call a point\vspace*{1pt} process \emph{$\alpha$-negatively associated}
(\emph{associated}) if
$\alpha^{(k+l)}(B_1 \times B_2) \leq(\geq) \alpha^{(k)}(B_1)\alpha
^{(l)}(B_2)$ for all $k,l \geq1$ and bounded Borel
$B_1 \times B_2 \subset(\mR^d)^k\times(\mR^d)^l$.

Negative association (association) implies $\alpha$-negative
association (association) which in turn implies $\alpha$-weaker
ordering with respect to the Poisson process with intensity measure
$\alpha$.

Even weaker notions of association come from looking at void
probabilities, and
we say that a point process $\Phi_1$ is \emph{$\nu$-weaker} than $\Phi
_2$ (denoted by $\Phi_1 \leq_{\nu-w} \Phi_2$) if
\[
\nu_1(B) = \mathbb{P} \bigl\{ \Phi_1(B) = \varnothing
\bigr\} \leq\mathbb{P} \bigl\{ \Phi_2(B) = \varnothing\bigr\} =
\nu_2(B)
\]
for all $B$ bounded Borel subsets.

Finally, we call a point process \emph{$\alpha$-sub-Poisson}
(\textit{super-Poisson})  if it is $\alpha$-weaker (stronger) than the Poisson
point process and similarly for \emph{$\nu$-sub-Poisson}
(\textit{super-Poisson}). A point process is \emph{weakly sub-Poisson}
(\textit{super-Poisson}) if it is both \mbox{$\alpha$- and} $\nu$-sub-Poisson
(super-Poisson).

It is known that negative association (association) implies the weak
sub-Poisson (super-Poisson) property. Other examples come from
perturbed lattices. For example, if the replication kernel $N$ is
hypergeometric or binomial and $X$ uniform, then the resulting
perturbed lattice $\Phi(N,X)$ is a weakly sub-Poisson point process.
One can also construct a sequence of perturbed lattices $\Phi(N_n,X),
n \geq1$ whose joint intensities and void probabilities monotonically
increase to that of the Poisson point\vspace*{1pt} process by choosing the
replication kernels $N_n$ to be distributed as $\operatorname{Bin}(n,\frac{1}{n})$.
On the other hand, negative binomial and geometric perturbation kernels
lead to weakly super-Poisson processes. Permanental point processes are
also weakly sub-Poisson. See \cite{Blaszczy14} for proofs and more
about stochastic ordering of point processes.

\subsection{Two technical lemmas}

We shall state some general results about Palm measures of these point
processes that we need later. The first lemma shows that negatively
associated point processes are ``stochastically stronger'' than their
Palm versions. This can also be viewed as a justification for the usage
of negative association as the defining property of sparse point
processes. The second shows that Palm versions of negatively associated
point processes also exhibit negative association. We state the results
in more generality than we need, since they seem to be of independent interest.
%

\begin{lem}
\label{lemNApalmvoid}
Let $\Phi$ be a negatively associated stationary point process in $\mR
^d$ of unit intensity and $F\dvtx\mR^{dn} \to\mR_+$ an increasing
bounded continuous function. Then for $B_1,\ldots,B_n$ disjoint
bounded Borel subsets and almost every $\mathbf{x}\in\mR^{dk}$,
%
\begin{eqnarray}
\label{equnlemma21} \qquad\E^!_{x_1,\ldots,x_k}\bigl(F\bigl(\Phi
(B_1),\ldots,\Phi(B_n)\bigr)\bigr) \leq\mathbb{E} \bigl\{F\bigl(
\Phi(B_1),\ldots,\Phi(B_n)\bigr) \bigr\}.
\end{eqnarray}
The above inequality will be reversed for an associated point process.
\end{lem}

\begin{pf} For $0<\epsilon< r$ we have
\begin{eqnarray*}
& & \mathbb{E} \bigl\{F\bigl(\Phi\bigl(B_1\setminus
B_{\mathbf{x}}(r)\bigr),\ldots,\Phi\bigl(B_n\setminus
B_{\mathbf{x}}(r)\bigr)\bigr) | \Phi\bigl(B_{x_i}(\epsilon)\bigr)
\geq1, 1 \leq i \leq k \bigr\}
\nonumber
\\
&&\qquad =  \frac{\mathbb{E} \{F(\Phi(B_1\setminus B_{\mathbf{x}}(r)),\ldots
,\Phi(B_n\setminus B_{\mathbf{x}}(r)))\prod_{i=1}^k\1
[\Phi(B_{x_i}(\epsilon)) \geq1] \} }{\mathbb{P} \{ \Phi
(B_{x_i}(\epsilon)) \geq1, 1 \leq i \leq k \}}
\\
&&\qquad \leq \frac{\mathbb{E} \{F(\Phi(B_1\setminus B_{\mathbf{x}}(r)),\ldots
,\Phi(B_n\setminus B_{\mathbf{x}}(r))) \} \mathbb
{E} \{\prod_{i=1}^k\1[\Phi(B_{x_i}(\epsilon)) \geq1] \} }{\mathbb{P} \{
\Phi(B_{x_i}(\epsilon)) \geq1, 1 \leq i
\leq k \}}
\\
&&\qquad =  \mathbb{E} \bigl\{F\bigl(\Phi\bigl(B_1\setminus
B_{\mathbf{x}}(r)\bigr),\ldots,\Phi\bigl(B_n\setminus
B_{\mathbf{x}}(r)\bigr)\bigr) \bigr\},
\end{eqnarray*}
where the inequality is due to the negative association property of
$\Phi$.

Sending first $\epsilon\to0$ and then $r\to0$, (\ref{equnlemma21})
follows immediately from \cite{Shirai03}, Lemma 6.3, and monotone convergence.
\end{pf}
%

\begin{lem}
\label{lemNAPalmNApp}
Let $\Phi$ be a negatively associated stationary point process in $\mR
^d$ of unit intensity. We shall also assume the existence of all the
joint intensities of the point process. Let $F\dvtx \mR^{dn} \to\mR_+$
and $G\dvtx \mR^{dm} \to\mR_+$ be increasing bounded continuous
functions. Then for $B_1,\ldots,B_{m+n}$ disjoint bounded Borel
subsets\vadjust{\goodbreak} and almost every $\mathbf{x}\in\mR^{d(k+l)}$,
\begin{eqnarray*}
& & \E^!_{\mathbf{x}}\bigl\{F\bigl(\Phi(B_1),\ldots,
\Phi(B_n)\bigr) G\bigl(\Phi(B_{n+1}),\ldots,
\Phi(B_{m+n})\bigr)\bigr\} \rho^{(k+l)}(\mathbf{x})
\\
&&\qquad\leq\E^!_{x_1,\ldots,x_k}\bigl\{F\bigl(\Phi(B_1),\ldots,\Phi
(B_n)\bigr)\bigr\}
\\
&&\qquad\quad{} \times\E^!_{x_{k+1},\ldots,x_{k+l}}\bigl\{G\bigl(\Phi(B_{n+1}),
\ldots,\Phi(B_{m+n})\bigr)\bigr\}
\\
&&\quad\qquad{}\times\rho^{(k)}(x_1,
\ldots,x_k)\rho^{(l)}(x_{k+1},\ldots,x_{k+l}).
\end{eqnarray*}
The above inequality will be reversed for an associated point process
$\Phi$.
\end{lem}

\begin{pf}
As in the proof of Lemma~\ref{lemNApalmvoid}, take $0<\epsilon<r$.
For notational simplicity, set $B^*=B\setminus B_{\mathbf{x}}(r)$ for
bounded Borel set $B$:
\begin{eqnarray*}
&& \E\bigl\{ F\bigl(\Phi\bigl(B^*_1\bigr),\ldots,\Phi
\bigl(B^*_n\bigr)\bigr) G\bigl(\Phi\bigl(B^*_{n+1}\bigr),
\ldots,\Phi\bigl(B^*_{m+n}\bigr)\bigr)
\\
&&\hspace*{98pt}| \Phi\bigl(B_{x_i}(\epsilon)\bigr) \geq1, 1 \leq i \leq(k+l) \bigr
\}
\\
&&\quad{}\times\P\bigl\{\Phi\bigl(B_{x_i}(\epsilon)\bigr) \geq1, 1
\leq i \leq(k+l) \bigr\}
\\
&&\qquad =\E\Biggl\{F\bigl(\Phi\bigl(B^*_1\bigr),\ldots,\Phi
\bigl(B^*_n\bigr)\bigr)\prod_{i=1}^k
\1\bigl[\Phi\bigl(B_{x_i}(\epsilon)\bigr) \geq1\bigr]
\\
&&\hspace*{44pt}{} \times G\bigl(\Phi\bigl(B^*_{n+1}\bigr),\ldots,\Phi
\bigl(B^*_{m+n}\bigr)\bigr)\prod_{i=1}^{l}
\1\bigl[\Phi\bigl(B_{x_{k+i}}(\epsilon)\bigr) \geq1\bigr] \Biggr\}
\\
&&\qquad \leq\frac{\mathbb{E} \{F(\Phi(B^*_1),\ldots,\Phi
(B_n^*))\prod_{i=1}^k\1[\Phi(B_{x_i}(\epsilon)) \geq1] \}
}{\mathbb{P} \{ \Phi(B_{x_i}(\epsilon)) \geq1, 1 \leq i \leq k
\}}
\\
&&\quad\qquad{} \times\frac{\mathbb{E} \{G(\Phi
(B^*_{n+1}),\ldots,\Phi(B^*_{m+n}))\prod_{i=1}^l\1[\Phi
(B_{x_{k+i}}(\epsilon)) \geq1] \} }{\mathbb{P} \{ \Phi
(B_{x_{k+i}}(\epsilon)) \geq1, 1 \leq i \leq l \}}
\\
&&\quad\qquad{} \times\mathbb{P} \bigl\{ \Phi\bigl(B_{x_i}(\epsilon)\bigr) \geq1, 1
\leq i \leq k \bigr\}\mathbb{P} \bigl\{ \Phi\bigl(B_{x_{k+i}}(\epsilon)
\bigr) \geq1, 1 \leq i \leq l \bigr\}
\\
&&\qquad =\mathbb{E} \bigl\{F\bigl(\Phi\bigl(B^*_1\bigr),\ldots,\Phi
\bigl(B_n^*\bigr)\bigr) \llvert\Phi\bigl(B_{x_i}(\epsilon)
\bigr) \geq1, 1 \leq i \leq k \bigr\}
\\
&&\quad\qquad{}\times
\mathbb{E} \bigl\{G\bigl(\Phi\bigl(B^*_{n+1}\bigr),
\ldots,\Phi\bigl(B^*_{m+n}\bigr)\bigr) | \Phi\bigl(B_{x_{k+i}}(
\epsilon)\bigr) \geq1, 1 \leq i \leq l \bigr\}
\\
&&\quad\qquad{}\times
\mathbb{P} \bigl\{
\Phi\bigl(B_{x_i}(\epsilon)\bigr) \geq1, 1 \leq i \leq k \bigr\}
\mathbb{P} \bigl\{ \Phi\bigl(B_{x_{k+i}}(\epsilon)\bigr) \geq1, 1 \leq i
\leq l \bigr\},
\end{eqnarray*}
where the inequality is due to the negative association of $\Phi$. As
in the previous proof, the conditional expectations in the first and
last expressions converge to the respective Palm expectations as
$\epsilon\to0$. Since $\Phi$ is a simple point process, after
dividing by $\llvert B_0(\epsilon)\rrvert^{k+l}$ on both sides, the
product of the
probability terms in the last line converges to $\rho^{(k)}(x_1,\ldots
,x_k)\rho^{(l)}(x_{k+1},\ldots,x_{k+l})$ while\vspace*{1pt} the probability term
in the first line converges to $\rho^{(k+l)}(\mathbf{x})$ as $\epsilon
\to0$. Complete the proof by sending $r \to0$.
\end{pf}
%

\section{Subgraph and component counts in random geometric graphs}\label{secsubgraphrgg}

Recall that for a point set $\Phi$ and radius $r > 0$, the geometric
graph $G(\Phi,r)$ is defined as the graph with vertex set $\Phi$ and
edge-set $\{(X,Y)\dvtx \llVert X-Y\rrVert\leq r\}$.
We shall work with restrictions of $\Phi$ to a sequence of increasing
windows $W_n = [\frac{-n^{1/d}}{2},\frac{n^{1/d}}{2}]^d$, along with
a radius regime $\{r_n > 0\}_{n \geq1}$, setting
$\Phi_n:= \Phi\cap W_n$. The choice of the radius regime will impact
on the asymptotic properties of the geometric graph
when the points of $\Phi$ are those of a point process.

Let $\Gamma$ be a connected graph on $k$ vertices. In this section we
shall be interested in how often $\Gamma$
appears (up to graph isomorphisms) in a sequence of geometric graphs
$G_n=G(\Phi_n,r_n)$, and how often among such appearances it is
actually isomorphic to a component of $G_n$; namely, it is a $\Gamma
$-component of $G_n$. For graphs built over Poisson and i.i.d.
processes, we know from \cite{Penrose03}, Chapters~3, 13, that
no $\Gamma$-components exist when $n(r_n^d)^{k-1} \to0$ ($\llvert
\Gamma\rrvert=
k$), but that they do appear when
$n(r_n^d)^{k-1} \to\infty$. The $\Gamma$-components continue to
exist even when $r_n^d = o(\log n)$ and vanish when $r_n^d = \omega
(\log n)$, which is the threshold for connectivity of the graph.

In this section, we shall show, among other things, that the threshold
for formation of $\Gamma$-components for negatively associated
processes with $r_n\to0$ is $n(r_n^d)^{k-1}f^k(r_n) \to\infty$, for
functions $f^k$ which typically satisfy
$f^k(r)\to0$ as $r\to0$, and so is higher than in the Poisson case.
These components continue to exist even when $r_n^d \to\beta> 0$. The
threshold for the vanishing of components will be treated in the next section.

The reader should try to keep this broader picture in mind as she wades
through the various limits of this section.

\subsection{Some notation and a start}
%
As above, let $\Gamma$ be a connected graph on $k$ vertices, $k \geq
1$ and $\{x_1,\dots,x_k\}$ a collection of $k$
points in $\mR^d$. Introduce the
(indicator) function
$h_{\Gamma}\dvtx\mR^{dk}\times\mR_+\to\{0,1\}$ by
%
\begin{eqnarray}
\label{eqndefnha3.1} h_{\Gamma}(\mathbf{x},r) &:=& \1\bigl[G\bigl(
\{x_1,\ldots,x_k\},r\bigr) \simeq\Gamma\bigr],
\end{eqnarray}
where $\simeq$ denotes graph isomorphism and $\1$ is the usual
indicator function. For a fixed sequence $\{r_n\}$ set
%
\begin{eqnarray}
\label{eqndefnh} h_{\Gamma,n}(\mathbf{x}) &:=& h_{\Gamma}(
\mathbf{x},r_n),
\end{eqnarray}
and, for $r=1$, write
%
\begin{eqnarray}
\label{eqndefnhb} h_{\Gamma}(\mathbf{x}) &:= & h_{\Gamma}(\mathbf{x},1).
\end{eqnarray}

Moving now to the random setting, in which $\Phi$ is a simple point
process with $k$th intensities $\rho^{(k)}$,
we say that $\Gamma$ is \emph{a feasible subgraph} of $\Phi$ if
%
\begin{eqnarray*}
\int_{(\mR^d)^k} h_\Gamma(\mathbf{x})\rho^{(k)}(
\mathbf{x}) \,d\mathbf{x} > 0.
\end{eqnarray*}
Thus $\Gamma$ is a feasible subgraph of $\Phi$ if the $\alpha^{(k)}$
measure of finding a copy of it (up to graph isomorphism) in $G(\Phi
,1)$ is positive. For many of our examples of point processes and
graphs $\Gamma$, feasibility will hold because $\rho^{(k)}(\mathbf{x})
> 0$ a.e. or at least on a large enough set.

We shall be interested in the the number of $\Gamma$-subgraphs,
$G_n(\Phi,\Gamma)$, and number of $\Gamma$-components, $J_n(\Phi,\Gamma
)$, of $\Phi_n$, which are defined as follows:
\begin{eqnarray}
\label{eqnGn} G_n(\Phi,\Gamma) &:=& \frac{1}{k!}\sum
_{X \in\Phi
_n^{(k)}}h_{\Gamma,n}(X),
\nonumber\\[-8pt]\\[-8pt]\nonumber
J_n(\Phi,\Gamma) &:=& \frac{1}{k!}\sum
_{X \in\Phi
_n^{(k)}}h_{\Gamma,n}(X)\1\bigl[\Phi_n
\bigl(B_X(r_n)\bigr) = k\bigr].
\nonumber
\end{eqnarray}
We shall now make a small digression to clarify the terminology. In the
terminology of \cite{Penrose03}, $G_n(\Phi,\Gamma)$ is referred to
as the number of induced $\Gamma$-subgraphs of $G(\Phi,r)$ and not
the number of $\Gamma$-isomorphic subgraphs. However, it is easy to
see that the latter is a finite linear combination of the number of
induced subgraphs of the same order. We shall be considering only
induced subgraphs in this article and hence shall chose to omit the
adjective {\it induced}.

Note that $J_n$ considers graphs based on vertices in $\Phi_n$ only,
namely, all vertices that lie in $W_n$. Such a graph, however, may have
vertices in the complement of~$W_n$, provided the points are within
distance a $r_n$ of $W_n$, and so actually be part of something larger.
To account for this boundary effect, we introduce an additional
variable, which does not count such ``boundary crossing'' graphs. This
is given by
%
\begin{equation}
\tJ_n(\Phi,\Gamma):= \frac{1}{k!}\sum
_{X \in\Phi
_n^{(k)}}h_{\Gamma,n}(X)\1\bigl[\Phi\bigl(B_X(r_n)
\bigr) = k\bigr].
\end{equation}
We shall see later that in the sparse and thermodynamic regimes, the
differences between $J_n$ and $\tJ_n$ disappear in asymptotic results.
Nevertheless, both are needed for the proofs.

The key ingredient in obtaining asymptotics for sub-graph counts and
component counts are the following closed-form expressions, which are
immediate consequences of the Campbell--Mecke formula:
%
%
\begin{eqnarray}
\mathbb{E} \bigl\{G_n(\Phi,\Gamma) \bigr\} & = & \frac
{1}{k!}
\int_{W_n^k}h_{\Gamma,n}(\mathbf{x}) \rho^{(k)}(
\mathbf{x}) \,d\mathbf{x}, \label{eqnExpGn}
\\
\qquad\mathbb{E} \bigl\{J_n(\Phi,\Gamma) \bigr\} & = &
\frac
{1}{k!}\int_{W_n^k}h_{\Gamma,n}(\mathbf{x})
\P^!_{\mathbf{x}}\bigl\{\Phi_n\bigl(B_{\mathbf{x}}(r_n)
\bigr) = 0\bigr\} \rho^{(k)}(\mathbf{x}) \,d\mathbf{x}. \label{eqnExpJn}
\end{eqnarray}
Much of the remainder of this section is based on obtaining asymptotic
expressions for these integrals in terms of basic
point process parameters in the sparse and thermodynamic regimes, as
well as looking at bounds on variances.
We shall consider the connectivity regime only in the following section
on Betti numbers.
Our results here extend those of \cite{Penrose03}, Chapter~3, for
Poisson and i.i.d. processes, and
the general approach of the proofs is thus similar.

\subsection{Sparse regime: \texorpdfstring{$r_n\to0$}{rnto0}}

The intuition behind the following theorem is that in the sparse regime
it is difficult to find $\Gamma$-subgraphs in a random geometric
graph, and even more unlikely that any such subgraph will have another
point of the point process near it. This implies that almost all [in
the sense made precise by~(\ref{eqnexpsubgrsp})] such subgraphs
will actually be a component of the full graph, disconnected from other
components.
%

\begin{theorem}
\label{propexpsubgrsp}
Let $\Phi$ be a stationary point process in $\mR^d$ of unit intensity
and $\Gamma$ be a feasible connected graph\vspace*{1pt} of $\Phi$ on $k$ vertices.
Let $\rho^{(k)} $ be almost everywhere continuous. Assume that $\rho
^{(k)}(0,\ldots,0)=0$, and that there exist functions $f^k_{\rho}\dvtx
\mR_+ \to\mR_+$ and $g^k_{\rho}\dvtx (B_0(k))^k \to\mR_+$ such that
\[
\rho^{(k)}(r\y) = \Theta\bigl(f^k_{\rho}(r)\bigr)
\quad\mbox{and}\quad \lim_{r \to0}\frac{\rho^{(k)}(r\y)}{f^k_{\rho}(r)} =
g^k_{\rho
}(
\y), %
\]
for all $\y$ of the form $\y= (0,y_2,\ldots,y_k)$. Further, assume
that $f^{k+1} = O(f^k )$ as $r \to0$ and $g^k_{\rho} $ is almost
everywhere continuous.
Let $r_n \to0$. Then
%
%
\begin{eqnarray}\label{eqnexpsubgrsp}
\lim_{n \to\infty} \frac{\mathbb{E} \{G_n(\Phi,\Gamma
) \} }{n r_n^{d(k-1)} f^k(r_n)} & = & \lim_{n \to\infty}
\frac
{\mathbb{E} \{J_n(\Phi,\Gamma) \} }{n r_n^{d(k-1)}
f^k(r_n)} \nonumber
\\
& = & \mu_0(\Phi,\Gamma)
\\
&:= & \cases{ 1, &\quad$k = 1$,
\vspace*{5pt}\cr
\displaystyle\frac{1}{k!}\int_{\mR^{d(k-1)}}h_{\Gamma}(
\y)g^k_{\rho}(\y) \,d \y, &\quad$k \geq1$.}
\nonumber
\end{eqnarray}
If $\rho^{(k)}(0,\ldots,0) > 0$, then the same result holds with
$f^k_{\rho} \equiv1$ and $g^k_{\rho} \equiv\rho^{(k)}(0,\break \ldots,0)$.
\end{theorem}

Before turning to the proof of the theorem, we shall make a few points
about its conditions, and provide some examples. As before, we are
assuming that all point processes are normalized to have unit intensity.
%

\begin{remark}\label{re3.2}
(1) Note that the theorem does not guarantee the positivity
of $\mu_0(\Phi,\Gamma)$.\vspace*{1pt}

(2) $f^1(r) \equiv1$ for all stationary point processes of
unit intensity since, in this case, $\rho^{(1)} \equiv1$.

(3) It is easy to check that if $\Phi$ is $\alpha
$-negatively associated or $\alpha$-super-Poisson, then the condition
$f^{k+1} = O(f^k )$ as $r \to0$ is satisfied.

(4) In the case $\rho^{k}(0,\ldots,0) = 0$ for $k\geq2$,
even if we cannot find appropriate $f^k $ or $g^k_{\rho}$,
it is still true that $\mathbb{E} \{G_n(\Phi,\Gamma) \}
= o(n r_n^{d(k-1)})$.

(5)
If $\Phi$ is only $\mZ^d$-stationary (as is the case with perturbed
lattices), then it will be clear from the proof that
(\ref{eqnexpsubgrsp}) still holds, but with
\[
\mu_0(\Phi,\Gamma):= \frac{1}{k!}\int_{[0,1]^d}
\int_{\mR^{d(k-1)}}h_{\Gamma}(x,\y) g^k_{\rho}(x,
\y) \,dx \,d\y. %
\]
%

(6)
For a homogeneous Poisson point process, the theorem holds with $f^k
\equiv1$ and $g^k_{\rho} \equiv1$, recovering \cite{Penrose03},
Proposition 3.1.

(7)
If\vspace*{1pt} $\Phi\geq_{\al-w} \Phi_{(1)}$, then for all $k \geq1$, $\rho
^{(k)} \geq\rho_{(1)}^{(k)} \equiv1$ and hence $f^k \equiv1$ and
also $\mu_0(\Phi,\Gamma) > 0$. Examples of point processes in this
class are all super-Poisson perturbed lattices and permanental point processes.

(8)
For\vspace*{1pt} a perturbed lattice $\Phi$ with perturbation kernel $N \in\{
0,\ldots,K\}$ a.s., $\rho^{(k)}(0,\ldots,0) > 0$ if and only if $k
\leq K$. In this case, $\mu_0(\Phi,\Gamma) > 0$ for a connected
graph $\Gamma$ on $k$ vertices. For connected graphs $\Gamma$ on $k$
vertices with $k > K$, $n r_n^{-d(k-1)} \mathbb{E} \{G_n(\Phi,\Gamma)
\} \to0$. For sub-Poisson perturbed lattices, the
existence of $f^k $ depends on the perturbation kernel. However, for
high values of $k$, it is clear that the scaling for sub-Poisson
perturbed lattices will differ significantly from that of the Poisson case.

(9)
From \cite{Nazarov12}, Theorem 1.1, for the zeroes of Gaussian entire
function and calculations similar to \cite{Ben09}, Theorem 4.3.10, for
the Ginibre point process, one can check that in both cases
\[
\rho^k(x_1,\ldots,x_k) = \Theta\biggl(\prod
_{i < j} \llVert x_i-x_j
\rrVert^2 \biggr). %
\]
Hence $f^k(r) = \Theta(r^{k(k-1)})$ for these processes.
\end{remark}

\begin{pf*}{Proof of Theorem 3.1}
We shall prove the theorem for $k \geq2$. The case $k = 1$ follows
easily by making a few notational changes to the general case. We start
with the convergence of $\mathbb{E} \{G_n(\Phi,\Gamma)
\} $.
In the expression for $\mathbb{E} \{G_n(\Phi,\Gamma) \}
$ in~(\ref{eqnExpGn}), make the change of variable $x_i = x_1 +
r_ny_i$ for $i \geq2$ and then use stationarity of the point process
to obtain
\begin{eqnarray*}
&& \mathbb{E} \bigl\{G_n(\Phi,\Gamma) \bigr\}
\\
&&\qquad  =
\frac{r_n^{d(k-1)}}{k!}\int_{W_n}\int_{(r_n^{-1}(W_n-x))^{k-1}}h_{\Gamma,n}(x
\mathbf{1}+r_n\y)
\rho^{(k)}(x\mathbf{1}+r_n\y
) \,dx \cdots dy_k
\\
&&\qquad  =  \frac{r_n^{d(k-1)}}{k!}\int_{W_n}\int
_{(r_n^{-1}(W_n-x))^{k-1}}h_{\Gamma,n}(r_n
\y)\rho^{(k)}(r_n\y) \,d x \cdots dy_k
\\
&&\qquad  \leq \frac{r_n^{d(k-1)}}{k!}\int_{W_n}\int
_{\mR
^{d(k-1)}}h_{\Gamma}(\y)\rho^{(k)}(r_n
\y) \,dx \cdots dy_k
\\
&&\qquad  =  \frac{nr_n^{d(k-1)}}{k!}\int_{\mR^{d(k-1)}}h_{\Gamma}(\y)
\rho^{(k)}(r_n\y) \,d\y.
\end{eqnarray*}
Since $\Gamma$ is a connected graph, $h_{\Gamma} \equiv0$ outside
$(B_0(k))^{k-1}$, and hence the preceding integral is finite. Further
for all $x \in W_{(n^{1/d}-2k)^d}$, it follows that
\[
B_O(k) \subset(W_n-x) \subset r_n^{-1}(W_n-x)
\]
for large $n$. Hence for large enough $n$,
\begin{eqnarray}
\mathbb{E} \bigl\{G_n(\Phi,\Gamma) \bigr\} & \geq&
\frac
{r_n^{d(k-1)}}{k!}\int_{W_{(n^{1/d}-2k)^d}}\int
_{(r_n^{-1}(W_n-x))^{k-1}}h_{\Gamma}(
\y)\rho^{(k)}(r_n\y) \,dx \cdots dy_k
\nonumber
\\
& = & \frac{r_n^{d(k-1)}}{k!}\int_{W_{(n^{1/d}-2k)^d}}\int_{\mR
^{d(k-1)}}h_{\Gamma}(
\y)\rho^{(k)}(r_n\y) \,dx \cdots d y_k
\nonumber
\\
& = & \frac{((n^{1/d}-2k)^d)r_n^{d(k-1)}}{k!}\int_{\mR
^{d(k-1)}}h_{\Gamma}(\y)
\rho^{(k)}(r_n\y) \,d\y.
\nonumber
\end{eqnarray}
Thus, as $n\to\infty$,
\[
\frac{\mathbb{E} \{G_n(\Phi,\Gamma) \}
}{nr_n^{d(k-1)}} \sim\frac{1}{k!}\int_{\mR^{d(k-1)}}h_{\Gamma}(
\y)\rho^{(k)}(r_n\y) \,d\y. %
\]
Note that we can restrict the range of integration in the above
equation to $B_0(k)$.
Since ${\rho^{(k)}(r_n\y)}/{f^k(r_n)} = g^k_{\rho}(\y)$ a.e. in
$B_0(k)$, and $g^k_{\rho} $ is bounded (as it is continuous) in
$B_0(k)$, we can use the Lebesgue dominated convergence theorem to show
that, as $n\to\infty$,
\[
\frac{\mathbb{E} \{G_n(\Phi,\Gamma) \}
}{nr_n^{d(k-1)}f^k(r_n)} \to\mu_0(\Phi,\Gamma).
\]
This proves the convergence of expected number of $\Gamma$-subgraphs.

We shall now show that the normalized expected numbers of components
and subgraphs are asymptotically
equivalent for small enough radii. This will complete the proof of the theorem.

Using the lower bound of $1 - \Phi(B_X(r_n))$ for the void term in
$J_n $ [see (\ref{eqnGn})], we obtain the following lower bound for
$J_n $:
\begin{eqnarray*}
J_n(\Phi,\Gamma) &\geq& G_n(\Phi,\Gamma) -
\frac{1}{k!}\sum_{X \in
\Phi_n^{(k)}}h_{\Gamma,n}(X)\Phi
\bigl(B_X(r_n)\bigr)
\\
&=& G_n(\Phi,\Gamma) -
E_n(\Phi,\Gamma).
\end{eqnarray*}
Since $J_n \leq G_n $, we only need to show that $\frac{\mathbb{E}
\{E_n(\Phi,\Gamma) \} }{nr_n^{d(k-1)}f^k(r_n)} \to0$.
From the Campbell--Mecke formula, we have
\[
\mathbb{E} \bigl\{E_n(\Phi,\Gamma) \bigr\} = \frac{1}{k!}\int
_{W_n^k}h_{\Gamma,n}(\mathbf{x})\E^!_{\mathbf{x}}\bigl\{
\Phi\bigl(B_{\mathbf{x}}(r_n)\bigr)\bigr\}\rho^{(k)}(
\mathbf{x}) \,d\mathbf{x}. %
\]
From \cite{Shirai03}, Lemma 6.4, we know that $\rho^{!(1)}_{\mathbf
{x}}(y) = \frac{\rho^{(k+1)}(\mathbf{x},y)}{\rho^{(k)}(\mathbf{x})}$.
Now applying the Campbell--Mecke formula for $\E^!_{\mathbf{x}} $ in
the above equation, we find that
\[
\mathbb{E} \bigl\{E_n(\Phi,\Gamma) \bigr\} = \frac{1}{k!}\int
_{W_n^k \times B_{\mathbf{x}}(r_n)}h_{\Gamma,n}(\mathbf{x})\rho^{(k+1)}(
\mathbf{x},y) \,d\mathbf{x} \,dy. %
\]
Now apply the change of variables $x_i = x_1 + r_ny_i$ for $i \geq2$,
$y = r_ny$ and proceed as in the case of $\mathbb{E} \{G_n
\} $ to see that, for large enough $n$,
\[
\mathbb{E} \bigl\{E_n(\Phi,\Gamma) \bigr\} \leq\frac{n
r_n^{dk}}{k!}
\int_{\mR^{d(k-1)} \times B_0(k+1)}h_{\Gamma}(\y)\rho^{(k+1)}(r_n
\y,r_ny) \,d\y \,dy,
\]
where the additional factor of $r_n^d$ is due to the $y$ variable.
Dividing by $nr_n^{d(k-1)}f^k(r_n)$ and bounding $h_{\Gamma}$ by 1, we have
\[
\frac{\mathbb{E} \{E_n(\Phi,\Gamma) \} }{n
r_n^{d(k-1)}f^k(r_n)} \leq\frac{r_n^df^{k+1}(r_n)}{k!f^k(r_n)} \int
_{B_0(k)^{k-1} \times B_0(k+1)}
\frac{\rho^{(k+1)}(r_n\y,r_ny)}{f^{k+1}(r_n)} \,dy \,d\y.
\]
Since $f^{k+1}(r) = O(f^k(r))$ by assumption, $\frac{\mathbb{E}
\{E_n(\Phi,\Gamma) \} }{n r_n^{d(k-1)}f^k(r_n)} \to0$
and hence
\[
\frac{\mathbb{E} \{J_n(\Phi,\Gamma) \} }{n
r_n^{d(k-1)}f^k(r_n)} \to\mu_0(\Phi,\Gamma),
\]
as required.
\end{pf*}

The following corollary follows easily from the ordering of the joint
intensities of the point processes.
%

\begin{cor}
\label{corcompsubgrlts}
Let $\Phi_i, i=1,2$, be two stationary point processes and
$f^k_{\rho_i}, g^k_{\rho_i}$ correspond\vspace*{1pt} to the functions of Theorem~\ref{propexpsubgrsp}. If $\Phi_1 \leq_{\al-w} \Phi_2$, then
\mbox{$f^k_{\rho_1} \leq f^k_{\rho_2} $}. If $f^k_{\rho_1} \equiv f^k_{\rho
_2} $, then $g^k_{\rho_1} \leq g^k_{\rho_2} $, and hence\vspace*{1pt}
$\mu_0(\Phi_1,\Gamma) \leq\mu_0(\Phi_2,\Gamma)$ for a connected
graph $\Gamma$ that is feasible for both $\Phi_1$ and $\Phi_2$.
\end{cor}
%

\subsection{Thermodynamic regime: \texorpdfstring{$r_n^d\to\beta$}{rndtobeta}}

%
\begin{theorem}
\label{propexpsubgrth}
Let $\Phi$ be a stationary point process in $\mR^d$ of unit intensity
and $\Gamma$ be a feasible connected graph of $\Phi$ on $k$ vertices.
Assume\vspace*{1pt} that $\rho^{(k)} $ is almost everywhere continuous, and let
$r_n^d \to\beta> 0$ and $\y= (0,y_2,\ldots,y_k)$. Then
%
\begin{eqnarray}
&& \lim_{n \to\infty} \frac{\mathbb{E} \{G_n(\Phi,\Gamma) \} }{n} \nonumber
\\
&&\qquad = \mu_{\beta}(\Phi,
\Gamma)\label{eqnexpsubgrth}
\\
&&\qquad := \cases{1, &\quad$k = 1$,
\vspace*{5pt}\cr
\displaystyle\frac{\beta^{k-1}}{k!} \int
_{\mR^{d(k-1)}}h_{\Gamma}(\y)\rho^{(k)}\bigl(
\beta^{1/d} \y\bigr) \,d\y, &\quad$k \geq2$,}\nonumber
\\
&& \lim_{n \to\infty} \frac{\mathbb{E} \{J_n(\Phi,\Gamma
) \} }{n} \nonumber
\\
&&\qquad = \gamma_{\beta}(\Phi,\Gamma) 
\\
&&\qquad :=
\cases{\P^{!}_{O} \bigl\{\Phi
\bigl(B_{O}\bigl(\beta^{1/d}\bigr)\bigr) = 0 \bigr\}, &
\quad$k = 1$,
\vspace*{5pt}\cr
\displaystyle\frac{\beta^{k-1}}{k!}\int_{\mR^{d(k-1)}}h_{\Gamma}(
\y)\rho^{(k)}\bigl(\beta^{1/d} \y\bigr),
\vspace*{5pt}\cr
\hspace*{58pt}{}\times
\P^{!}_{\beta^{1/d} \y} \bigl\{\Phi\bigl(B_{\beta^{1/d} \y}\bigl(
\beta^{1/d}\bigr)\bigr) = 0 \bigr\} \,d\y, &\quad$k \geq2$.}\hspace*{-30pt}\nonumber
\end{eqnarray}
If $\Phi$ is a negatively associated point process with $\mathbb
{P} \{ \Phi(B_{\mathbf{x}}(\beta^{1/d})) = 0 \} >
0$ for almost every $\mathbf{x}
\in B_0(\beta^{1/d}k)^k$, then $\gamma_{\beta}(\Phi,\Gamma
) > 0$.
\end{theorem}

Again, before turning to the proof, we make some observations about the theorem:
\begin{longlist}[(2)]
\item[(1)]
The positivity of $\gamma_{\beta}(\Phi,\Gamma)$ is not immediate.
For an example in which this does not hold,
let $\Phi_0$ be a Poisson point process of unit intensity in $\mR^d$,
$\Phi_i, i \geq1$ i.i.d.\vspace*{1pt} copies of the point process of $4$
i.i.d. uniformly distributed points in $B_O(\beta^{1/d}/2)$,
and define the Cox point process,
\[
\Phi:= \bigcup_{X_i \in\Phi_0} \{X_i +
\Phi_i\}. %
\]
Clearly, for all $X \in\Phi$, $\mathbb{P} \{ \Phi(B_X(\beta
^{1/d})) \geq4 \} = 1$.

Now take $r_n^d \equiv\beta$ and $\Gamma$ a triangle, and note that
$J_n(\Phi,\Gamma) = 0$ for all $n \geq1$ and
so $\gamma_{\beta}(\Phi,\Gamma) = 0$, even though all the
assumptions of Theorem~\ref{propexpsubgrth} are satisfied.

\item[(2)]
As in Corollary~\ref{corcompsubgrlts}, $\Phi_1 \leq_{\al-w}
\Phi_2$ implies that $\mu_{\beta}(\Phi_1,\Gamma) \leq\mu_{\beta
}(\Phi_2,\Gamma)$. However, as the previous example shows, the
situation for $\gamma_{\beta}(\Phi,\Gamma)$ is somewhat more complicated.

\item[(3)]
If $\llvert\Gamma\rrvert= 1$, then $J_n(\Phi,\Gamma)$ is the number
of isolated
nodes in the Boolean model of
balls of radii $\beta$ centered on the points of $\Phi$. The Palm
measure of a determinantal point process is also determinantal and in
particular, for the Ginibre process, $\rho^{!(1)}(z) = 1 - e^{-\llVert
z\rrVert
^2}$. Using this explicit structure, it can be shown that, for small
enough $\beta$,
\[
\gamma_{\beta}(\Phi_{\mathrm{Gin}},\Gamma) \geq1 - \pi
\beta^2 + \pi\bigl(1 - e^{-\beta^2}\bigr) > 1 - \pi
\beta^2 + O\bigl(\pi^2 \beta^4\bigr) =
\gamma_{\beta}(\Phi_{\mathrm{Poi}},\Gamma),
\]
and hence the inequality for the $\gamma_{\beta}$ could be reversed
in the thermodynamic regime for even negatively associated point
processes as compared to the sparse regime.
\end{longlist}

\begin{pf*}{Proof of Theorem \ref{propexpsubgrth}}
Since the proof here is similar to the preceding one, we shall not give
all the details, and again, we shall only bother with the case $k \geq
2$. Starting with (\ref{eqnExpGn}) and (\ref{eqnExpJn}), the
proof follows similar lines to that of Theorem~\ref
{propexpsubgrsp}. The difference is that $r_n^{d(k-1)} \to\beta
^{k-1}$ and $\rho^{(k)}(r_n \y) \to\rho^{(k)}(\beta^{1/d}
\y)$, and so there is no need for additional scaling. For the
convergence of $J_n$, one first shows the convergence of $\tJ_n$ using
similar techniques to those in the proof of Theorem~\ref
{propexpsubgrsp}. Then note that
\[
\tJ_n(\Phi,\Gamma) \leq J_n(\Phi,\Gamma) \leq\tJ
_n(\Phi,\Gamma) + G_n(\Phi_n/
\Phi_{(n^{1/d}-(k+1)r_n)^d},\Gamma).
\]
The rightmost term in the upper bound accounts for the boundary
effects, and by arguments similar to those in the proof of Theorem~\ref
{propexpsubgrsp}, it is easy to see that
\[
\mathbb{E} \bigl\{G_n(\Phi_n/\Phi_{(n^{1/d}-(k+1)r_n)},
\Gamma) \bigr\} = O\bigl(\llvert W_n/W_{(n^{1/d}-(k+1)r_n)^d}\rrvert\bigr) =
O\bigl(n^{(d-1)/d}\bigr).
\]
More importantly for us, this expectation is $o(n)$, and so of lower
order than $\tJ_n(\Phi,\Gamma)$. Thus $\mathbb{E} \{J_n(\Phi,\Gamma) \}
/n$ also converges to $\gamma_{\beta}(\Phi,\Gamma
)$. Since $r_n^d \to\beta> 0$, the void probability term in $J_n$ is
not necessarily degenerate.

The positivity of $\gamma_{\beta}(\Phi,\Gamma)$ for negatively
associated point processes is an easy corollary of Lemma~\ref
{lemNApalmvoid}. We need only note that
\[
F\bigl(\Phi(B)\bigr) = \1\bigl[\Phi(B) = 0\bigr] = \bigl(1-\Phi(B)\bigr
) \vee0
\]
is a decreasing bounded continuous function and hence
\[
\P^{!}_{\mathbf{x}}\bigl(\Phi\bigl(B_{\mathbf{x}}\bigl(
\beta^{1/d}\bigr)\bigr) = 0\bigr) \geq\mathbb{P} \bigl\{ \Phi
\bigl(B_{\mathbf{x}}\bigl(\beta^{1/d}\bigr)\bigr) = 0 \bigr\} > 0
\]
for a.e. $\mathbf{x}\in B_0(k)^k$. This completes the proof.
\end{pf*}
%

%

\subsection{Variance bounds for the sparse and thermodynamic regimes}\label{secphasetransition}
The crux of the second moment bounds lies in the fact that, up to
constants, variances are essentially
bounded above (below) by expectations for negatively associated
(associated) point processes. [It is simple to check that $\mathsf
{Var} (\cdot ) = \Theta(\mathbb{E} \{\cdot \} )$
for the Poisson process, which is both negatively associated and
associated; cf. \cite{Penrose03}, Chapter~3.]
We, however, shall need to extend these inequalities to graph
variables, and this is the content of this section.
%

\begin{theorem}[(Covariance bounds in sparse regime)]\label{propcovbdssp}
Let $\Gamma$ and $\Gamma_0$ be two feasible connected graphs on $k$
and $l$ ($k \geq l\geq2$) vertices, respectively, for a stationary
point process $\Phi$ with almost everywhere continuous joint
densities. Let
$\Phi$ satisfy the assumptions of Theorem~\ref{propexpsubgrsp}
and assume that the $f^j$ and $g^j_{\rho} $ exist for all $j \leq
k+l$. Further, let $r_n \to0$ and $\mu_0(\Phi,\Gamma) > 0$.
\begin{longlist}[(2)]
\item[(1)] If $\Phi$ is $\alpha$-negatively associated, then
\[
\mathsf{Cov} \bigl( G_n(\Phi,\Gamma),G_n(\Phi,
\Gamma_0) \bigr) = O\bigl(\mathbb{E} \bigl\{G_n(\Phi,
\Gamma) \bigr\} \bigr).
\]
\item[(2)] If $\Phi$ is $\alpha$-associated, then
\[
\mathsf{Cov} \bigl( G_n(\Phi,\Gamma),G_n(\Phi,
\Gamma_0) \bigr) = \Omega\bigl(\mathbb{E} \bigl\{G_n(
\Phi,\Gamma) \bigr\} \bigr).
\]
\end{longlist}
\end{theorem}

\begin{pf}
We shall prove the result for $\alpha$-negatively associated processes
and $k \geq2$. The $\alpha$-associated case follows by reversing
the inequality sign in (\ref{equn-crucial}) below, and the case $k =
1$ needs a few simple notational changes. We shall again use the
Campbell--Mecke formula to obtain closed-form expressions for the
second moments and then perform a similar analysis as in the proof of
Theorem~\ref{propexpsubgrsp} to obtain the asymptotics.

For $j \leq l$ and $\mathbf{x}= (x_1,\ldots,x_{k+l-j})$, in analogy to
(\ref{eqndefnha3.1})--(\ref{eqndefnhb}), define
\begin{eqnarray*}
h_{\Gamma,\Gamma_0,j}(\mathbf{x}) &:=& h_{\Gamma}(x_1,\ldots
,x_k)h_{\Gamma
_0}(x_1,\ldots,x_j,x_{k+1},\ldots,x_{k+l-j}),
\\
h_{\Gamma,\Gamma_0,j,n} (\mathbf{x}) &:=& h_{\Gamma,n}(x_1,\ldots
,x_k)h_{\Gamma
_0,n}(x_1,\ldots,x_j,x_{k+1},\ldots,x_{k+l-j}).
\end{eqnarray*}
%
Then
%
\begin{eqnarray}\label{equn-crucial}
&& \mathbb{E} \bigl\{G_n(\Phi,\Gamma)G_n(\Phi,\Gamma_0) \bigr\}\nonumber
\\
&&\qquad = \mathbb{E} \biggl\{\sum_{\mathcal{X},\mathcal{Y} \subset\Phi
,\llvert\mathcal{X}\rrvert= k, \llvert\mathcal{Y}\rrvert= l}
h_{\Gamma,n}(\mathcal{X})h_{\Gamma_0,n}(\mathcal{Y}) \biggr\}\nonumber
\\
&&\qquad  =  \sum_{j = 0}^l\mathbb{E} \biggl\{
\sum_{\mathcal{X},\mathcal{Y} \subset\Phi,\llvert\mathcal{X}\rrvert= k, \llvert\mathcal
{Y}\rrvert= l,\llvert\mathcal{X}
\cap\mathcal{Y}\rrvert= j} h_{\Gamma,n}(\mathcal{X})h_{\Gamma
_0,n}(\mathcal{Y}) \biggr\}\nonumber
\\
&&\qquad  =  \sum_{j=0}^l \frac{1}{j!(k-j)!(l-j)!}
\int_{W_n^{k+l-j}}h_{\Gamma,\Gamma_0,j,n}(\mathbf{x})\rho
^{(k+l-j)}(\mathbf{x}) \,d\mathbf{x}\nonumber
\\
&&\qquad \leq \sum_{j=1}^l \frac{1}{j!(k-j)!(l-j)!}
\int_{W_n^{k+l-j}}h_{\Gamma,\Gamma_0,j,n}(\mathbf{x})\rho
^{(k+l-j)}(\mathbf{x}) \,d\mathbf{x}
\\
&&\quad\qquad{}+ \frac{1}{k!l!}\int_{W_n^k}\int
_{W_n^l}h_{\Gamma,\Gamma_0,0,n}(\mathbf{x})\rho^{(k)}(x_1,\ldots,x_k)
\nonumber
\\
&&\hspace*{100pt}{}\times\rho^{(l)}(x_{k+1},\ldots,x_{k+l})
\,dx_1 \cdots dx_{k+l}
\nonumber
\\
&&\qquad =  \sum_{j=1}^l \frac{1}{j!(k-j)!(l-j)!}
\int_{W_n^{k+l-j}}h_{\Gamma,\Gamma_0,j,n}(\mathbf{x})\rho
^{(k+l-j)}(\mathbf{x}) \,d\mathbf{x}
\nonumber
\\
&&\quad\qquad{} + \mathbb{E} \bigl\{G_n(\Phi,\Gamma) \bigr\}
\mathbb{E} \bigl\{G_n(\Phi,\Gamma_0) \bigr\},
\nonumber
\end{eqnarray}
%
where the inequality is due to the $\alpha$-negative association
property. Thus using similar arguments as in the proof of Theorem~\ref
{propexpsubgrsp} and setting $\y= (0,y_2,\ldots,y_{k+l-j})$, we have
%
\begin{eqnarray}
\qquad&& \mathsf{Cov} \bigl( G_n(\Phi,
\Gamma),G_n(\Phi,\Gamma_0) \bigr)
\nonumber
\\
&&\qquad \leq\sum_{j=1}^l \frac{1}{j!(k-j)!(l-j)!}
\int_{W_n^{k+l-j}}h_{\Gamma,\Gamma_0,j,n}(\mathbf{x})\rho
^{(k+l-j)}(\mathbf{x}) \,d\mathbf{x}
\nonumber
\\
&&\qquad \sim\sum_{j=1}^l \frac
{nr_n^{d(k+l-j-1)}f^{k+l-j}(r_n)}{j!(k-j)!(l-j)!}
\int_{\mR
^{d(k+l-j-1)}}h_{\Gamma,\Gamma_0,j}(\y)g^{(k+l-j)}_{\rho}(
\y) \,d \y
\\
&&\qquad = O\bigl(nr_n^{d(k-1)}f^{k}(r_n)\bigr)\nonumber
\\
&& \qquad= O\bigl(\mathbb{E} \bigl\{G_n(\Phi,\Gamma) \bigr
\} \bigr),
\nonumber
\end{eqnarray}
which is what we needed to show.
\end{pf}

Unlike the sparse regime, subgraph counts and component counts have
different limits in the thermodynamic regime and hence we need variance
bounds on component counts in the thermodynamic regime.

%
\begin{theorem}[(Variance bounds in the thermodynamic regime)]\label
{propvarbdsth}
Let $\Phi$ be a negatively associated stationary point process in $\mR
^d$ of unit intensity and $\Gamma$ be a feasible connected graph of
$\Phi$ on $k$ vertices. Assume that $\rho^{(k)} $ is almost
everywhere continuous.
Let $r_n^d \to\beta> 0$ and $\gamma_{\beta}(\Phi,\Gamma) > 0$.
Then we have that
\[
\mathsf{Var} \bigl(\tJ_n(\Phi,\Gamma) \bigr) = O\bigl(\mathbb{E}
\bigl\{J_n(\Phi,\Gamma) \bigr\} \bigr).
\]
\end{theorem}

\begin{pf}
First, write
\begin{eqnarray*}
\tJ_n(\Phi,\Gamma)^2 &=& \tJ_n(\Phi,\Gamma)
+ \sum_{X,Y \subset
\Phi_n, \llvert X\rrvert= \llvert Y\rrvert= k} h_{\Gamma
,n}(X)h_{\Gamma,n}(Y)
\\
&&\hspace*{126pt}{} \times\1\bigl[\Phi\bigl(B_X(r_n)
\bigr) = \Phi\bigl(B_Y(r_n)\bigr) = 0\bigr].
\end{eqnarray*}
By the Campbell--Mecke formula,
\begin{eqnarray*}
&& \mathbb{E} \bigl\{\tJ_n(\Phi,\Gamma)^2 \bigr\}
\\
&&\qquad = \mathbb{E} \bigl\{\tJ_n(\Phi,\Gamma) \bigr\}
\\
&&\hspace*{30pt}{}+
\frac{1}{(k!)^2}\int_{W_n^k
\times W_n^k} h_{\Gamma,n}(
\mathbf{x})h_{\Gamma,n}(\y)
\1\bigl[G\bigl(\{\mathbf{x},\y\};r_n\bigr)\mbox{ is
disconnected}\bigr]
\\
&&\hspace*{103pt}{} \times\P^!_{\mathbf{x},\y}\bigl\{\Phi\bigl(B_{\mathbf{x}}(r_n)
\bigr) = \Phi\bigl(B_{\y}(r_n)\bigr) = 0\bigr\}
\rho^{(2k)}(\mathbf{x},\y)\,d\mathbf{x} \,d \y.
\end{eqnarray*}
Thus
\begin{eqnarray*}
\mathsf{Var} \bigl(\tJ_n(\Phi,\Gamma) \bigr) & = & \mathbb{E} \bigl
\{\tJ_n(\Phi,\Gamma) \bigr\}
\\
&&{} +\frac{1}{(k!)^2}\int_{W_n^k \times W_n^k} h_{\Gamma,n}(\mathbf
{x})h_{\Gamma,n}(\y) Q_n(\mathbf{x},\y)\,d\mathbf{x} \,d\y,
\end{eqnarray*}
where
\begin{eqnarray*}
Q_n(\mathbf{x},\y) &:=& \1\bigl[G\bigl(\{\mathbf{x},\y\};r_n
\bigr)\mbox{ is disconnected}\bigr]
\\
&&{} \times\P^!_{\mathbf{x},\y}\bigl\{\Phi\bigl(B_{\mathbf{x}}(r_n)
\bigr) = \Phi\bigl(B_{\y}(r_n)\bigr) = 0\bigr\}
\\
&&{} \times\rho^{(2k)}(\mathbf{x},\y) - \P^!_{\mathbf{x}}\bigl\{\Phi
\bigl(B_{\mathbf{x}}(r_n)\bigr) = 0\bigr\}
\\
&&{} \times\P^!_{\y}\bigl\{\Phi\bigl(B_{\y}(r_n)
\bigr) = 0\bigr\} \rho^{(k)}(\mathbf{x}) \rho^{(k)}(\y).
\end{eqnarray*}
Choose $n$ large enough so that $r_n \leq\beta^{1/d} + \frac
{1}{4}$. For such an $n$ and negatively associated $\Phi$, we know
from Lemma~\ref{lemNAPalmNApp} that $Q_n(\mathbf{x},\y) \leq0$
for all $\mathbf{x},\y$ such that the set distance $d_S(\mathbf{x},\y):=
\inf_{i,j}\llVert x_i-y_j\rrVert> 3\beta^{1/d}$. Thus, we
have that
\begin{eqnarray*}
\mathsf{Var} \bigl(\tJ_n(\Phi,\Gamma) \bigr) &\leq& \mathbb{E}
\bigl\{\tJ_n(\Phi,\Gamma) \bigr\}
\\
&&{} + \frac{1}{(k!)^2}\int
_{W_n^k \times W_n^k} h_{\Gamma,n}(\mathbf{x})h_{\Gamma,n}(\y)
\\
&&\hspace*{72pt}{} \times Q_n(\mathbf{x},\y)\1\bigl[d(\mathbf
{x},\y) \leq3\beta^{1/d}\bigr] \,d\mathbf{x} \,d\y.
\end{eqnarray*}
From Theorem~\ref{propexpsubgrth}, we know that $\mathbb{E}
\{\tJ_n(\Phi,\Gamma) \} = \Theta(n)$, and using similar
methods as in that theorem, one can show that the latter term in the
above equation is of $O(n)$. Combining the two, we get that $\mathsf
{Var} (\tJ_n(\Phi,\Gamma) ) = O(n)$.
\end{pf}
%

\subsection{Phase transitions in the sparse and thermodynamic regimes}
\label{sec2phasetransition2}

So far, we have concentrated on the asymptotic behavior of the
expectations of the numbers of different types of subgraphs that appear
in the random graph associated with a point process. In this section we
shall combine expectations on first and second moment to obtain results
about these numbers themselves, looking at probabilities that they are
nonzero, as well as $L_2$ and almost sure results about growth and
decay rates. The main theorem of this section is the following:
%

\begin{teo}
\label{teophtralnegassoc}
Let $\Phi$ be a stationary point process with almost everywhere
continuous joint densities and $\Gamma$ a feasible connected graph for
$\Phi$ on $k$ vertices.
\begin{enumerate}[(1)]
\item[(1)] Let $\Phi$ satisfy the assumptions of Theorem~\ref
{propexpsubgrsp} with $\mu_0(\Phi,\Gamma) > 0$. Let $r_n \to0$.
\begin{longlist}[(a)]
\item[(a)] If\vspace*{1pt} $n r_n^{d(k-1)} f^k(r_n) \to0$,\footnote{Note that
neither this assumption nor the one in (1)(b) can hold for $k = 1$,
as $f^1(r) \equiv1$. Hence the statements do not say anything in these
two cases.} then $\mathbb{P} \{ G_n(\Phi,\Gamma) \geq1
\} \to0$.
\item[(b)] If $\Phi$ is $\alpha$-negatively associated and $n
r_n^{d(k-1)} f^k(r_n) \to\beta$ for some $0 < \beta< \infty$, there
there exists a finite $C$ (dependent
on the process but not on $\Gamma$)
for which
\begin{eqnarray*}
\lim_{n \to\infty} \mathbb{P} \bigl\{ J_n(\Phi,\Gamma)
\geq1 \bigr\} \geq\biggl[1 + \frac{C}{\beta\mu_0(\Phi,\Gamma
)} \biggr]^{-1}.
\end{eqnarray*}

\item[(c)] If $\Phi$ is $\alpha$-negatively associated and $n
r_n^{d(k-1)} f^k(r_n) \to\infty$, then
\begin{eqnarray*}
\frac{J_n(\Phi,\Gamma)}{n r_n^{d(k-1)} f^k(r_n)} \stackrel{L_2} {\to}
\mu_0(\Phi,
\Gamma).
\end{eqnarray*}
\end{longlist}

\item[(2)] Let $\Phi$ be a negatively associated point process
satisfying the assumptions of Theorem~\ref{propexpsubgrth} with
$\gamma_\beta(\Phi,\Gamma) > 0$. Let $r_n^d \to\beta$. Then
%
%
\begin{equation}
\label{eqnL2convergencethermo} \frac{J_n(\Phi,\Gamma)}{n} \stackrel
{L_2} {\to}
\gamma_{\beta
}(\Phi,\Gamma).
\end{equation}
\end{enumerate}
\end{teo}

\begin{pf}
The proof for part~(1)(a) follows from Markov's inequality and Theorem
\ref{propexpsubgrsp}. The proof of (1)(b) is based on the following
second moment bound:
\begin{eqnarray*}
\mathbb{P} \bigl\{ J_n(\Phi,\Gamma) \geq1 \bigr\} &\geq&
\frac
{(\mathbb{E} \{J_n(\Phi,\Gamma) \} )^2}{\mathbb{E}
\{J_n(\Phi,\Gamma)^2 \} }
\\
&\geq& \frac{(\mathbb{E} \{J_n(\Phi,\Gamma) \}
)^2}{\mathbb{E} \{G_n(\Phi,\Gamma)^2 \} }
\\
&\geq& \biggl[\frac{(\mathbb{E} \{G_n(\Phi,\Gamma) \}
)^2}{(\mathbb{E} \{J_n(\Phi,\Gamma) \} )^2} + \frac
{\mathsf{Var} (G_n(\Phi,\Gamma) ) }{(\mathbb{E}
\{J_n(\Phi,\Gamma) \} )^2} \biggr]^{-1}.
\end{eqnarray*}
Now, by applying Theorem~\ref{propcovbdssp}, we obtain that there
exists a $C > 0$ for which
\[
\mathbb{P} \bigl\{ J_n(\Phi,\Gamma) \geq1 \bigr\} \geq\biggl[{
\frac{(\mathbb{E} \{G_n(\Phi,\Gamma) \}
)^2}{(\mathbb{E} \{J_n(\Phi,\Gamma) \} )^2} + \frac
{C}{\mathbb{E} \{G_n(\Phi,\Gamma) \} }} \biggr]^{-1}. %
\]
Under the assumptions of (1)(b), $\mathbb{E} \{G_n(\Phi,\Gamma
) \} $ converges to $\beta\mu_0(\Phi,\Gamma)$, while the
first term in the square brackets converges to 1 by Theorem~\ref
{propexpsubgrsp}.

For (1)(c), observe that
\begin{eqnarray*}
\mathsf{Var} \bigl(J_n(\Phi,\Gamma) \bigr) &\leq&\mathsf{Var}
\bigl(G_n(\Phi,\Gamma) \bigr)
\\
&&{} + 2\mathbb{E} \bigl\{G_n(
\Phi,\Gamma) \bigr\} \mathbb{E} \bigl\{E_n(\Phi,\Gamma) \bigr\} -
\bigl(\mathbb{E} \bigl\{E_n(\Phi,\Gamma) \bigr\} \bigr)^2,
\end{eqnarray*}
where $E_n(\Phi,\Gamma)$ is as defined in the proof of Theorem~\ref
{propexpsubgrsp}. From the proofs of Theorems \ref{propexpsubgrsp}~and~\ref{propcovbdssp}, it follows that
\[
\mathsf{Var} \bigl(J_n(\Phi,\Gamma) \bigr) = O\bigl(\mathsf{Var}
\bigl(G_n(\Phi,\Gamma) \bigr) \bigr) = O\bigl(n r_n^{d(k-1)}
f^k(r_n)\bigr),
\]
which completes the proof for this case.

We now prove part~2 in a similar fashion. In fact, it follows easily
from Theorem~\ref{propvarbdsth} and the relation between $J_n$ and
$\tJ_n$ noted in the proof of Theorem~\ref{propexpsubgrth}. More
specifically, as $n \to\infty$,
\begin{eqnarray*}
\mathsf{Var} \biggl(\frac{\tJ_n(\Phi,\Gamma)}{n} \biggr) & \to& 0
\end{eqnarray*}
and
\begin{eqnarray*}
\mathbb{E} \biggl\{\frac{\llVert J_n(\Phi,\Gamma) - \tJ
_n(\Phi,\Gamma)\rrVert}{n} \biggr\} & = & \frac{\mathbb{E} \{
G_n(\Phi_n/\Phi_{(n^{1/d}-(k+1)r_n)^d},\Gamma) \} }{n} \to0.
\end{eqnarray*}
Thus, we have that $\frac{\tJ_n(\Phi,\Gamma)}{n} \stackrel{P}{\to
} \gamma_{\beta}(\Phi,\Gamma)$ and $\frac{\llVert J_n(\Phi,\Gamma) -
\tJ_n(\Phi,\Gamma)\rrVert}{n} \stackrel{P}{\to} 0$ as $n \to\infty$.
\end{pf}

Since $J_n $ is a $K_d$-Lipschitz functional of counting measures for a
constant $K_d$ depending only on the dimension $d$ (see \cite
{Penrose03}, Proof of Theorem~3.15), the result in (\ref
{eqnL2convergencethermo}) above can be strengthened to a
concentration inequality for stationary determinantal point processes
by using the concentration inequality in \cite{Pemantle14}, Theorem~3.6. Further, (\ref{eqnL2convergencethermo}) can also
be extended to a strong law for ergodic point processes via the methods
used in \cite{Yogesh14}, Lemma 3.2.
%

\subsection{Extension to subcomplex counts}\label{secsubcomplex}

The earlier section was concerned about subgraph and component counts
but, as will be seen later, the techniques can be adapted to the
analysis of wider classes of functionals. One specific class of
functionals for which we shall explicitly state the asymptotics are
subcomplex counts. These will be used in the next section. While
asymptotics for Vietoris--Rips complexes can be derived using those of
subgraph counts, we shall need the results of this section to derive
the corresponding asymptotics for \v{C}ech complexes. We shall need a
few definitions before stating these results.

Let $\cK$ and $\cL$ be two complexes with vertex-sets $V_1$ and $V_2$,
respectively. A~function $f\dvtx V_1 \to V_2$ is called a \textit{simplicial
map} if $[f(v_1),\ldots,f(v_k)]$ is a face of $\cL$ whenever
$[v_1,\ldots,v_k]$ is a $k$-face of $\cK$. If $f$ is a bijection and
$f^{-1}$ is also a simplicial map, $f$ is said to be a \textit{simplicial
isomorphism}. If there exists a simplicial isomorphism between two
complexes $\cK$ and $\cL$, then we write $\cK\simeq\cL$.

Let $\Delta$ be a complex on $k$ vertices ($k \geq1$) such that its
$1$-dimensional skeleton (i.e., the underlying graph) is connected (as
a graph), and let $\{x_1,\dots,x_k\}$ be a collection of $k$ points in
$\mR^d$. As in the graph case, introduce the (indicator) function $\h
_{\Delta}\dvtx\mR^{dk}\times\mR_+\to\{0,1\}$ defined by
%
\begin{eqnarray}
\label{eqndefnha} \h_{\Delta}(\mathbf{x},r) &:=& \1\bigl[C\bigl(
\{x_1,\ldots,x_k\},r\bigr) \simeq\Delta\bigr],
\end{eqnarray}
where $\simeq$ denotes simplicial isomorphism, and $C$ was defined in
Definition~\ref{defcechcomplex}. Let $\Phi$ be a simple stationary
point\vspace*{1pt} process and $r_n, n \geq1$ be a sequence of radii. As before,
setting $\h_{\Delta}(\mathbf{x}):= \h_{\Delta}(\mathbf{x},1)$, we
call $\Delta$ a \textit{feasible subcomplex} of $\Phi$ if
\begin{eqnarray*}
\int_{(\mR^d)^k} \h_{\Delta}(\mathbf{x})\rho^{(k)}(
\mathbf{x}) \,d\mathbf{x} > 0.
\end{eqnarray*}

We can define an (induced) subcomplex count for the \v{C}ech complex
on the point process $\Phi_n$ as follows:
\[
\tC_n(\Phi,\Delta):= \frac{1}{k!}\sum
_{X \in\Phi_n^{(k)}}\h_{\Delta}(X,r_n).
\]
Also of interest is the number of isolated $\Delta$ subcomplexes of
the \v{C}ech complex on the point process $\Phi_n$, defined as follows:
\[
\tC^*_n(\Phi,\Delta):= \frac{1}{k!}\sum
_{X \in\Phi_n^{(k)}}\h_{\Delta}(X,r_n)\1\bigl[
\Phi_n\bigl(B_X(r_n)\bigr) = k\bigr].
\]

For the sake of brevity and to avoid repetition, we shall not provide
the proofs of the following two theorems, as they are a simple
extension of the proofs of Theorems~\ref{propexpsubgrsp},~\ref
{propexpsubgrth} and~\ref{teophtralnegassoc}; see also the
explanation before (\ref{eqncriticalpts}) in Section~\ref{secmorsergc}.

%
%
\begin{theorem}
\label{propexpcompsp}
Let $\Phi$ be a stationary point process satisfying the assumptions of
Theorem~\ref{propexpsubgrsp} and $\Delta_k$ be a feasible
connected complexes of $\Phi$ on $k$ vertices. Let $r_n \to0$. Then
%
%
\begin{eqnarray}\label{eqnexpcompsp}
\lim_{n \to\infty} \frac{\mathbb{E} \{\tC_n(\Phi,\Delta
_k) \} }{n r_n^{d(k-1)} f^k(r_n)} & = & \lim_{n \to\infty}
\frac{\mathbb{E} \{\tC^*_n(\Phi,\Delta_k) \} }{n
r_n^{d(k-1)} f^k(r_n)}\nonumber
\\
& = & \tilde{\mu}_0(\Phi,\Delta_k)
\\
&:= & \cases{1, &\quad$k = 1$,
\vspace*{5pt}\cr
\displaystyle\frac{1}{k!}\int_{\mR^{d(k-1)}}h_{\Delta_k}(
\y)g^k_{\rho}(\y) \,d\y, &\quad$k \geq2$.}
\nonumber
\end{eqnarray}

If $\rho^{(k)}(0,\ldots,0) > 0$, then the same result holds with
$f^k_{\rho} \equiv1$ and $g^k_{\rho} \equiv\rho^{(k)}(0,\ldots,0)$.

If $\Phi$ is $\alpha$-negatively associated, $\tilde{\mu}_0(\Phi,\Delta
_k) > 0$ and $n r_n^{d(k-1)} f^k(r_n) \to\infty$, then
\begin{eqnarray*}
\frac{\tC^*_n(\Phi,\Delta_k)}{n r_n^{d(k-1)} f^k(r_n)} \stackrel{L_2}
{\to} \tilde{\mu}_0(
\Phi,\Delta_k).
\end{eqnarray*}
\end{theorem}

%
\begin{theorem}
\label{propexpcompth}
Let $\Phi$ be a stationary point process in $\mR^d$ of unit intensity
and~$\Delta$ be a feasible connected complex of $\Phi$ on $k$
vertices. Assume that $\rho^{(k)} $ is almost everywhere continuous,
and let $r_n^d \to\beta> 0$ and $\y= (0,y_2,\ldots,y_k)$. Then
%
\begin{eqnarray}
&& \lim_{n \to\infty} \frac{\mathbb{E} \{\tC_n(\Phi,\Delta) \} }{n} \nonumber
\\
&&\qquad =  \tilde{\mu}_{\beta}(\Phi,\Delta) 
\\
&&\qquad := \cases{1, &\quad$k = 1$,
\vspace*{5pt}\cr
\displaystyle\frac{\beta^{k-1}}{k!} \int
_{\mR^{d(k-1)}}\h_{\Delta}(\y)\rho^{(k)}\bigl(
\beta^{1/d} \y\bigr) \,d\y, &\quad$k \geq2$,}\nonumber
\\
&& \lim_{n \to\infty} \frac{\mathbb{E} \{\tC^{*}_n(\Phi,\Delta) \} }{n}\nonumber
\\
&&\qquad = \tilde{\gamma}_{\beta}(\Phi,\Delta) 
\\
&&\qquad := \cases{\P^{!}_{O} \bigl\{\Phi
\bigl(B_{O}\bigl(\beta^{1/d}\bigr)\bigr) = 0 \bigr\}, &
\quad$k = 1$,
\vspace*{5pt}\cr
\displaystyle\frac{\beta^{k-1}}{k!}\int_{\mR^{d(k-1)}}
\h_{\Delta}(\y)\rho^{(k)}\bigl(\beta^{1/d} \y\bigr)
\vspace*{5pt}\cr
\hspace*{57pt}{}\times\P^{!}_{\beta^{1/d} \y} \bigl\{\Phi
\bigl(B_{\beta^{1/d} \y}\bigl(\beta^{1/d}\bigr)\bigr) = 0 \bigr\} \,d\y,
&\quad$k \geq2$.}\hspace*{-30pt}\nonumber
\end{eqnarray}
If $\Phi$ is a negatively associated point process and $\tilde{\gamma
}_{\beta}(\Phi,\Delta) > 0$, then
\begin{eqnarray*}
\frac{\tC^*_n(\Phi,\Delta_k)}{n} \stackrel{L_2} {\to} \tilde{
\gamma}_{\beta}(\Phi,\Delta_k).
\end{eqnarray*}

Further, if $\Phi$ is a negatively associated point process such that
for almost every $\mathbf{x}= (x_1,\ldots,x_k)
\in B_0(\beta^{1/d}k)^k$, $\mathbb{P} \{ \Phi(B_{\mathbf
{x}}(\beta^{1/d})) = 0 \} > 0$, then $\tilde{\gamma
}_{\beta}(\Phi,\Delta) > 0$.
\end{theorem}
%

\section{Betti numbers of random geometric complexes}\label{secbettirgc}
This is really the main section of the paper, giving, as it does,
results about the homology of random geometric complexes through their
Betti numbers. Despite this, it will turn out that, as mentioned
earlier, the hard work for the proofs has already been done in the
previous section.

We shall start with a review of the basic topological notions needed to
formulate our results, along with an explanation of the connections
between Betti numbers of random complexes, component numbers of random
geometric graphs and subcomplex counts. This connection was established
and exploited in \cite{Kahle11,Kahle13} to extract theorems for Betti
numbers from those for the component counts of random geometric graphs
and subcomplex counts.
%

\subsection{Topological preliminaries}\label{sectopprelims}

Recall that \v{C}ech and Vietoris--Rips complexes and their faces
were already defined at
Definitions~\ref{defcechcomplex} and~\ref{defripscomplex} in the
\hyperref[secintro]{Introduction}, and that the dimension of a face $\sigma$ is $\llvert
\sigma\rrvert-1$.
Recall also that the edges of the random geometric graph $G(\Phi,r)$
are the $1$-dimensional faces of $C(\Phi,r)$ or $R(\Phi,r)$.

Now, however, we require some additional terminology. The
Vietoris--Rips complex $R(\Phi_n,r)$ is also called the \textit{clique
complex} (or \textit{flag complex}) of $G(\Phi_n,r)$, as the faces are
cliques (complete subgraphs) of the $1$-dimensional faces. Let
$H_k(C(\Phi_n,r))$ and $H_k(R(\Phi_n,r))$,
respectively, denote the $k$th simplicial homology groups of the random
\v{C}ech and Vietoris--Rips complexes.
(We shall take our homologies over the field $\mZ_2$, but this will
not be important.)
In this section we shall be concerned with asymptotics for the Betti
numbers $\beta_k(C(\Phi_n,r))$ and $\beta_k(R(\Phi_n,r))$, (i.e., the
ranks of the homologies) and through them the appearance and
disappearance of homology groups.

Next, let $P_k$ be the $(k+1)$-dimensional cross-polytope in $\mR
^{k+1}$, containing the origin, and defined
to be the convex hull of the $2k+2$ points $\{\pm e_i\}$, where
$e_1,\ldots,e_{k+1}$ are the standard basis vectors of $\mR^{k+1}$.
The boundary of $P_k$, which we denote by $\tilde{O}_k$, is a
$k$-dimensional simplicial complex, homotopic to a \mbox{$k$-}dimensional
sphere. Let $O_k$ be the $1$-skeleton of $\tilde{O}_k$ that is, the
clique complex of the graph $O_k$ is $\tilde{O}_k$. In terms of
simplicial homology of the random Vietoris--Rips complexes, the
existence of subgraphs isomorphic to $O_k$ is the key to understanding
$k$-cycles, and so the $k$th homology.
In fact, from \cite{Kahle09}, Lemma 5.3, we know that, because the
Vietoris--Rips complex is a clique complex,
any nontrivial element of the $k$-dimensional homology $H_k(R(\Phi
_n,r))$ arises from a subcomplex on at least $2k+2$ vertices. If it has
only $2k+2$ vertices, then it will be isomorphic to $\tilde{O}_k$ and
the corresponding $1$-skeleton will be isomorphic to $O_k$.

Now let $\Gamma^j_k, j = 1,\ldots,n_k$ ($n_k < \infty$) be an
ordering of the different graphs that arise when extending a
($k+1$)-clique (i.e., a $k$-dimensional face) to a minimal (in terms
of the number of edges) connected subgraph on $2k+3$ vertices. Thus the
$\Gamma^j_k$ are all graphs on $2k+3$ vertices, having ${k+1\choose2}
+ k + 2$ edges.

Finally, for a given finite graph $\Gamma$, let $\tilde{G}(\Phi
_n,\Gamma)$ denote the number of subgraphs of $G(\Phi_n,r_n)$ that
are isomorphic to $\Gamma$. However, as explained in the discussion
after (\ref{eqnGn}), $\tilde{G}(\Phi_n,\Gamma)$ is a finite
linear combination of $G_n(\Phi,\Gamma')$'s with $\Gamma'$'s
being of the same order as $\Gamma$.

Then \cite{Kahle09}, Lemma 5.3, and a dimension bound in \cite
{Kahle11}, equation~(3.1), imply the following crucial inequality
linking Betti
numbers to component and subgraph counts in Vietoris--Rips complexes
for $k \geq1$ and for all $n \geq1$:
%
%
\begin{equation}
\label{eqnboundbettiVR} J_n(\Phi_n,O_k) \leq
\beta_k\bigl(R(\Phi_n,r_n)\bigr) \leq
J_n(\Phi,O_k) + \sum_{j=1}^{n_k}
\tilde{G}_n\bigl(\Phi,\Gamma^j_{k}\bigr).
\end{equation}

A related inequality holds for \v{C}ech complexes. Let $\tilde{\Gamma
}_k$ be the complex on $k$ vertices such that any $k-1$ vertices form a
$(k-1)$-face, but $\tilde{\Gamma}_k$ is not a $k$-face. Any
collection of vertices $X$ for which $G(X,r) \simeq\tilde{\Gamma}_k$
is said to form an \textit{empty} $(k-1)$-\textit{simplex}. Let $\tilde
{\Gamma
}'_k$ be the complex of a $(k-1)$-face with an extra edge attached
to two vertices and $\tilde{\Gamma}''_k$ be the graph of a
$(k-1)$-face with a path of length $2$ attached to one of the vertices.
Both $\tilde{\Gamma}'_k$ and $\tilde{\Gamma}''_k$ are
complexes of order $k+1$. The we have the following combinatorial
inequality from \cite{Kahle13}, equation~(5), for $k \in\{0,\ldots,d-1\}$ and
for all $n \geq1$:
%
%
\begin{eqnarray}\label{eqnboundbettiC}
\tC^*_n(\Phi,\tilde{\Gamma}_{k+2}) &\leq&
\beta_k\bigl(C(\Phi_n,r_n)\bigr)
\nonumber\\[-8pt]\\[-8pt]\nonumber
&\leq&
\tC^*_n(\Phi,\tilde{\Gamma}_{k+2}) + \tC_n\bigl(
\Phi,\tilde{\Gamma}'_{k+2}\bigr) + \tC_n\bigl(
\Phi,\tilde{\Gamma}''_{k+2}\bigr).
\end{eqnarray}
With these combinatorial inequalities in hand, we are now ready to
develop limit theorems for the Betti numbers of the random \v{C}ech and
Vietoris--Rips complexes (Section~\ref{secexpbetti}) as well as
find thresholds for vanishing and nonvanishing of homology groups
(Section~\ref{secvanthresholdhomgrp}).

\subsection{Expectations of Betti numbers}\label{secexpbetti}
We return now to the setting of a stationary point process $\Phi$ in
$\mR^d$ and the sequence of finite point processes $\Phi_n$. Our
results all follow quite easily from the corresponding limit theorems
in Section~\ref{secsubgraphrgg}, and we continue to use the
notation of that section without further comment.

The underlying heuristic is that in the sparse regime the order is
determined by the order of the minimal structure involved in forming
homology groups, which is $O_k$ for the random Vietoris--Rips complex
and $\Gamma_k$ for the random \v{C}ech complex. Using Theorem~\ref
{propexpsubgrsp} for the Vietoris--Rips complexes and Theorem~\ref
{propexpcompsp} for the \v{C}ech complexes, it is easy to see that
these are the leading order terms and that the $G$ and $\tilde G$ terms
in both (\ref{eqnboundbettiVR}) and (\ref{eqnboundbettiC}) are,
asymptotically, irrelevant. Hence, we have the following result.
%

%
\begin{theorem}[(Sparse regime: $r_n \to0$)]
\label{propbettinosp}
Let $\Phi$ be a stationary point process in $\mR^d$ satisfying the
assumptions in Theorem~\ref{propexpsubgrsp} for all $k \geq1$.
Let $r_n \to0$. Further, assume that $\tilde{\mu}_0(\Phi,\tilde
{\Gamma}_{k+2})>0$ for all $k \in\{0,\ldots,d-1\}$ and $\mu_0(\Phi,O_k)
> 0$ for all $k \geq1$. Then
\begin{eqnarray*}
\lim_{n \to\infty} \frac{\mathbb{E} \{\beta_k(C(\Phi
_n,r_n)) \} }{n r_n^{d(k+1)} f^{k+2}(r_n)} & = & \tilde{\mu
}_0(\Phi,\tilde{\Gamma}_{k+2}),\qquad k \in\{0,\ldots,d-1\},
\\
\lim_{n \to\infty} \frac{\mathbb{E} \{\beta_k(R(\Phi
_n,r_n)) \} }{n r_n^{d(2k+1)} f^{2k+2}(r_n)} & = & \mu_0(\Phi,O_k),\qquad k
\geq1.
\end{eqnarray*}
For $k = 0$, we have that
\begin{eqnarray}
\lim_{n \to\infty} \frac{\mathbb{E} \{\beta_0(C(\Phi
_n,r_n)) \} }{n} = \lim_{n \to\infty}
\frac{\mathbb{E}
\{\beta_0(R(\Phi_n,r_n)) \} }{n} = 1.
\nonumber
\end{eqnarray}
\end{theorem}

\begin{pf}
We start with the case $k \geq1$ and $k \in\{0,\ldots,d-1\}$ for the
Vietoris--Rips and \v{C}ech complexes, respectively. From Theorems
\ref{propexpsubgrsp} and~\ref{propexpcompsp}, the orders of
magnitude of the terms in (\ref{eqnboundbettiVR}) and (\ref
{eqnboundbettiC}) are as follows:
\begin{eqnarray*}
\mathbb{E} \bigl\{\tC^*_n(\Phi,\tilde{\Gamma}_{k+2}) \bigr
\} &=& \Theta\bigl(n r_n^{d(k+1)} f^{k+2}(r_n)
\bigr),
\\
\mathbb{E} \bigl\{\tC_n\bigl(\Phi,\tilde{\Gamma}'_{k+2}
\bigr) \bigr\} &=& \Theta\bigl(n r_n^{d(k+2)}
f^{k+3}(r_n)\bigr),
\\
\mathbb{E} \bigl\{\tC_n\bigl(\Phi,\tilde{\Gamma}''_{k+2}
\bigr) \bigr\} &=& \Theta\bigl(n r_n^{d(k+2)}
f^{k+3}(r_n)\bigr),
\\
\mathbb{E} \bigl\{J_n(\Phi,O_k) \bigr\} &=& \Theta
\bigl(n r_n^{d(2k+1)} f^{2k+2}(r_n)\bigr),
\\
\mathbb{E} \bigl\{\tilde{G}\bigl(\Phi,\Gamma^j_{k}\bigr)
\bigr\} &=& \Theta\bigl(n r_n^{d(2k+2)} f^{2k+3}(r_n)
\bigr), \qquad1 \leq j \leq n_k.
\end{eqnarray*}
Substituting these into (\ref{eqnboundbettiVR}) and (\ref
{eqnboundbettiC}), and
using the fact that the limits of $\mathbb{E} \{\tC^*_n(\Phi,\tilde
{\Gamma}_{k+2}) \} $ and $\mathbb{E} \{J_n(\Phi,O_k) \} $ are
explicitly known from Theorems~\ref{propexpsubgrsp} and~\ref{propexpcompsp}, completes the proof
of the theorem.

For the case $k = 0$, the bounds similar to (\ref{eqnboundbettiVR})
on $\beta_0$ and a similar argument will give the right asymptotics.
\end{pf}
Turning now to the thermodynamic regime, and applying the same
arguments as in the previous proof, but using Theorems~\ref{propexpsubgrth} and~\ref{propexpcompth} in place of Theorems
\ref{propexpsubgrsp} and~\ref{propexpcompsp}, we find that all
the terms in (\ref{eqnboundbettiVR}) and (\ref{eqnboundbettiC})
are of order $\Theta(n)$. This leads to the following
result.
%

\begin{theorem}[(Thermodynamic regime: $r_n^d \to\beta$)]
Let $\Phi$ be a stationary point process in $\mR^d$ satisfying the
assumptions in Theorem~\ref{propexpsubgrth} for all $k \geq1$.
Let $r_n^d \to\beta\in(0,\infty)$. Further, assume that $\tilde
{\gamma}_{\beta}(\Phi,\tilde{\Gamma}_k)>0$ for all $k \in\{
0,\ldots,d-1\}$ and $ \gamma_{\beta}(\Phi,O_k) > 0$ for all $k \geq
1$. Then, for all $k \geq0$,
\[
\mathbb{E} \bigl\{\beta_k\bigl(R(\Phi_n,r_n)
\bigr) \bigr\} = \Theta(n), %
\]
and for all $k \in\{0,\ldots,d-1\}$,
\[
\mathbb{E} \bigl\{\beta_k\bigl(C(\Phi_n,r_n)
\bigr) \bigr\} = \Theta(n).
\]
\end{theorem}

The above asymptotics have been strengthened to convergence and strong
laws in the recent preprint \cite{Yogesh14}. Further, we note without
proof that one can obtain ordering results for Betti numbers of $\al
-w$ ordered point processes in the sparse regime analogous to Corollary
\ref{corcompsubgrlts} but not in the thermodynamic regime.

\subsection{Thresholds for homology groups}\label{secvanthresholdhomgrp}
Our aim in this subsection is to establish results about the conditions
under which different homology groups appear and disappear in the
homology of random complexes. We shall need to treat \v{C}ech and
Vietoris--Rips complexes separately, and start with results on the
contractibility of these. We follow these with the key results of the
section, on thresholds for the appearance and
disappearance of homology groups. These results also show that $\gamma
$-weakly sub-Poisson point processes have lower vanishing thresholds
for given $\Gamma$-components. As a corollary to the results on \v
{C}ech complexes, we also obtain an asymptotic result on the behavior
of the Euler characteristic $\chi (C(\Phi,r))$.

Recall that there are a number of equivalent definitions for the Euler
characteristic. However, the most natural for us at this point is
%
\begin{eqnarray}
\chi \bigl(C(\Phi,r)\bigr):= \sum_{k \geq0}(-1)^k
\beta_k\bigl(C(\Phi,r)\bigr). \label{Eulerbybettiequn}
\end{eqnarray}
%

\begin{teo}[(Contractibility of \v{C}ech complexes)]\label{teocechcontract}
Let $\Phi$ be a stationary \mbox{$\gamma$-}weakly sub-Poisson point process.
Then there exists a $C_d > 0$ such that for $r_n \geq C_d(\log
n)^{1/d}$, w.h.p. $C(\Phi_n,r_n)$ is contractible and $\chi
(C(\Phi_n,r_n)) = 1$.
\end{teo}

\begin{pf}
We start with a proof of contractibility and then show that $\chi
(C(\Phi_n,r_n)) = 1$, w.h.p.
As in the proof of contractibility for Poisson \v{C}ech complexes in
\cite{Kahle11}, Theorem~6.1, we shall show that,
for our choice of $r_n$, the set $\bigcup_{X \in\Phi_n}B_X(r_n/2)$
covers $W_n$ w.h.p. Then the nerve theorem of \cite{Bjorner95},
Theorem~10.7, implies that the \v{C}ech complex is contractible
w.h.p.
Let $\mZ^d$ be the $d$-dimensional lattice, and let $Q_{z_i},1 \leq i
\leq N_n$ be an enumeration of the cubes of the scaled lattice $\frac
{r_n}{4\sqrt{d}}\mZ^d$ that are fully contained within $W_n$. If
every cube contains a point of $\Phi$, then $\bigcup_{X \in\Phi
_n}B_X(r_n/2)$ covers $W_n$. By the union bound,
\begin{eqnarray*}
\mathbb{P} \biggl\{ W_n \nsubseteq\bigcup
_{X \in\Phi_n}B_X(r_n/2) \biggr\} & \leq&
\sum_{i=1}^{N_n}\mathbb{P} \bigl\{
\Phi(Q_{z_i}) = 0 \bigr\}
\\
& \leq& N_n\mathbb{P} \biggl\{ \Phi_{(1)}
\biggl(B_O\biggl(\frac{r_n}{8\sqrt
{d}}\biggr)\biggr) = 0 \biggr\}
\\
&\leq& \frac{(4\sqrt{d})^dn}{r_n^d}e^{-({r_n}/({8\sqrt{d}}))^d},
\end{eqnarray*}
where $\Phi_{(1)}$ is the Poisson point process of unit intensity. All
that remains is to choose an appropriate $C_d > 0$ to complete the
proof of contractibility for general stationary $\gamma$-weakly
sub-Poisson point processes.

As for the proof of the statement about the Euler characteristic, the
following obvious bound suffices:
\begin{eqnarray*}
&&\mathbb{P} \biggl\{ W_n \subset\bigcup
_{X \in\Phi_n}B_X(r_n/2) \biggr\}
\\
&&\qquad\leq\mathbb{P} \bigl\{ \beta_0\bigl(C(\Phi_n,r_n)
\bigr) = 1, \beta_k\bigl(C(\Phi_n,r_n)\bigr)
= 0, k \geq1 \bigr\}
\\
&&\qquad\leq\mathbb{P} \bigl\{ \chi\bigl(C(\Phi_n,r_n)
\bigr) =1 \bigr\}.
\end{eqnarray*}\upqed
\end{pf}

With these results in hand, we can now use bounds (\ref
{eqnboundbettiVR}) and (\ref{eqnboundbettiC}) along with $L_2$
convergence results of Theorems~\ref{teophtralnegassoc},~\ref
{propexpcompsp} and~\ref{propexpcompth} to complete the picture
about vanishing and nonvanishing of homology groups of \v{C}ech
complexes and Vietoris--Rips complexes.
%

\begin{teo}[(Thresholds for \v{C}ech complexes)]\label{teothresholdhomgrpcech}
Let $\Phi$ be a stationary point process satisfying the assumptions on
its joint intensities $\rho^{(k)} $ as in Theorems~\ref
{propexpsubgrsp} and~\ref{propexpsubgrth} for all $k \geq1$.
Then the following statements hold:
\begin{enumerate}[(2)]
\item[(1)] Let $\Phi$ be a $\gamma$-weakly sub-Poisson point process.
\begin{longlist}[(a)]
\item[(a)] If
\[
r_n^{d(k+1)}f^{k+2}(r_n) = o
\bigl(n^{-1}\bigr)\quad\mbox{or}\quad r_n^d = \omega(
\log_n),
\]
then\vspace*{1pt} $\beta_k(C(\Phi_n,r_n)) = 0, k \in\{0,\ldots,d-1\}$, w.h.p.

\item[(b)] If $r_n^d = \omega(\log_n)$, then $\beta_0(C(\Phi_n,r_n)) =
1$, w.h.p.
\end{longlist}

\item[(2)] Let $\Phi$ be a negatively associated point process.
Further assume that $\tilde{\mu}_0(\Phi,\Gamma_k)>0$ and $\tilde
{\gamma}_{\beta}(\Phi,\Gamma_k) > 0$, both for all $k \in\{
0,\ldots,d-1\}$ and all $\beta> 0$.
\begin{longlist}[(a)]
\item[(a)] If
\[
r_n^{d(k+1)}f^{k+2}(r_n) = \omega
\bigl(n^{-1}\bigr)\quad\mbox{and}\quad r_n^d = O(1),
\]
then $\beta_k(C(\Phi_n,r_n)) \neq0, k \in\{0,\ldots,d-1\}$,
w.h.p.
\item[(b)] If $r_n^d = O(1)$, then $\beta_0(C(\Phi_n,r_n)) \neq0$,
w.h.p.
\end{longlist}
\end{enumerate}
%
\end{teo}

In the absence of a contractibility result for the Vietoris--Rips
complex, we are unable to estimate the second thresholds, where the
homology groups vanish. Thus we have the following less complete
picture for the Vietoris--Rips complex. Since $H_0(C(\Phi_n,r_n)) =
H_0(R(\Phi_n,r_n))$, we shall restrict ourselves to only $H_k(R(\Phi
_n,r_n)), k \geq1$, in the following theorem.

%
\begin{teo}[(Thresholds for Vietoris--Rips complexes)]\label
{teothresholdhomgrpvr}
Let $\Phi$ be a stationary point process satisfying the assumptions on
its joint intensities $\rho^{(k)} $ as in Theorems~\ref{propexpsubgrsp} and~\ref{propexpsubgrth} for all $k \geq1$.
Then the following statements hold for $k \geq1$:
\begin{longlist}[(2)]
\item[(1)] If
\[
r_n^{d(2k+1)}f^{2k+2}(r_n) = o
\bigl(n^{-1}\bigr),
\]
then $\beta_k(R(\Phi_n,r_n)) = 0$, w.h.p.
\item[(2)] Let $\Phi$ be a negatively associated point process.
Further assume that $\mu_0(\Phi,O_k)>0$ and $\gamma_{\beta}(\Phi,O_k) >
0$, both for all $k \geq1$ and all $\beta> 0$. If
\[
r_n^{d(2k+1)}f^{2k+1}(r_n) = \omega
\bigl(n^{-1}\bigr)\quad\mbox{and}\quad r_n^d = O(1),
\]
then $\beta_k(R(\Phi_n,r_n)) \neq0$, w.h.p.
\end{longlist}
%
\end{teo}

\subsection{Further results for the Ginibre process}\label{secGinibre}

Using the special structure of the Ginibre point process, we can
improve on the threshold results of the last section. The radius regime
for contractibility of \v{C}ech complexes over the Ginibre point
process and zeros of GEF can be made more precise, as more is known
about void probabilities in these cases. Once we have the
contractibility or connectivity results, the upper bounds on the
thresholds for vanishing of Betti numbers in this special case can be improved.
%

\begin{teo}[(Contractibility of \v{C}ech complexes)]\label{teocechcontractgin}
Let $\Phi$ be the Ginibre point process or zeros of GEF. Then there
exists a $C_d > 0$ (depending on the point process) such that for $r_n
\geq C_d(\log n)^{1/4}$, w.h.p. $C(\Phi_n,r_n)$ is
contractible. Hence, $\beta_0(C(\Phi_n,r_n)) = 1$, $\beta_k(C(\Phi
_n,r_n)) = 0$, $k \geq1$ and $\chi(C(\Phi_n,\break r_n)) = 1$ w.h.p. for
$r_n^2 = \omega(\sqrt{log n})$.
\end{teo}

\begin{pf}
The proof follows along similar lines as the proof of Theorem~\ref
{teocechcontract} except that in this case, the void probabilities
are of strictly lower order and so, the radius for contractibility as
well. More precisely, we know from \cite{Ben09}, Proposition~7.2.1 and
Theorem 7.2.3, that for the Ginibre point process and zeros of
GEF, $- \log(\mathbb{P} \{ \Phi(B_O(r)) = 0 \}) = \Theta
(r^4)$ as $r \to\infty$. All that remains is to substitute these
bounds into the proof of Theorem~\ref{teocechcontract} to derive the
corresponding results for the Ginibre point process and zeros of GEF.
\end{pf}

For Vietoris--Rips complexes, we do not have a contractibility result
for the Ginibre point proceses, but as a consequence of the upper
bounds for the Palm void probabilities, we can obtain upper bounds on
the threshold for the vanishing of the Betti numbers as well.

%
%
\begin{teo}[(Disappearence of homology groups for Vietoris--Rips
complexes)]\label{teoVRcontractgin}
Let $\Phi$ be the Ginibre point process. Then there exists a $C_{d,k}
> 0$ such that for $r_n \geq C_{d,k}(\log n)^{1/4}$, we have
that w.h.p. $\beta_k(R(\Phi_n,r_n)) = 0$, $k \geq1$.
\end{teo}

The proof uses the discrete Morse theoretic approach (see \cite
{Forman02}) similar to that of~\cite{Kahle11}, Theorem 5.1, and the
reader is referred to that proof and the Appendix in~\cite{Kahle11}
for missing details. As in \cite{Kahle11}, Theorem 5.1, our proof
actually shows topological $k$-connectivity, though we do not state it
here explicitly to avoid defining further topological notions.

\begin{pf*}{Proof of Theorem \ref{teoVRcontractgin}}
As the point process is simple and stationary, index the points in
$\Phi$ as $X_1,X_2,\ldots$ such that $\llVert X_1\rrVert< \llVert
X_2\rrVert< \llVert X_3\rrVert<\cdots.$ Define $V$ to be the
collection of pairs of simplices
$(V_1,V_2)$, $V_1 \subset V_2$ with
\[
V_1 = [X_{i_1},\ldots,X_{i_k}]\quad\mbox{and}\quad
V_2 = [X_{i_0},X_{i_1},\ldots,X_{i_k}],
\]
where $i_0 < i_1 < \cdots< i_k$. In words, we pair a simplex with
another simplex of codimension $1$ in the original simplex only if the
additional point is closer to the origin than the rest. A simplex that
is not in $V$ is said to be a critical simplex. Let $C_k$ be the number
of critical $k$-simplices of $V$. From discrete Morse theory, we know
that $ \beta_k(R(\Phi_n,r_n)) \leq C_k$. Thus, we only need to show
that $\mathbb{E} \{C_k \} \to0$ for all $k \geq1$, for
an appropriate choice of radii.

A $k$-simplex $\X= [X_{i_0},\ldots,X_{i_k}]$ where $i_0 < i_1 <
\cdots< i_k$ is critical only if
\[
\Phi_n\Biggl(\bigcap_{j=0}^kB_{X_{i_j}}(r)
\cap B_O\bigl(\llVert X_{i_0}\rrVert\bigr) \Biggr) = \{
X_{i_0}\}.
\]
Hence, using Campbell--Mecke formula for the first inequality, then
using \cite{Kahle11}, Lemma 5.3---that is, for a critical $k$-simplex
as above, there exists an $\epsilon_d > 0$ and $x \in\mR^d$ such that
\[
B_x(\epsilon_d) \subset\bigcap
_{j=0}^kB_{X_{i_j}}(r) \cap B_O
\bigl(\llVert X_{i_0}\rrVert\bigr)
\]
---and Lemma~\ref{lempalm-voidGinibre} for the second inequality and
finally $\rho^{(k)} \leq1$ for the last inequality, we find that
\begin{eqnarray*}
\mathbb{E} \{C_k \} & \leq& \int_{W_n^{k+1}} 1[\mathbf{x}\mbox{ is a
simplex}]1\bigl[\llVert x_{i0}\rrVert< \llVert
x_{i_1}\rrVert< \cdots< \llVert x_{i_k}\rrVert\bigr]
\\
&&\hspace*{26pt}{} \times\P^!_{\mathbf{x}}\Biggl\{\Phi_n\Biggl(\bigcap
_{j=0}^kB_{x_{i_j}}(r) \cap B_O
\bigl(\llVert x_{i_0}\rrVert\bigr)\Biggr) = 0\Biggr\}
\rho^{(k)}(\mathbf{x}) \,d\mathbf{x}
\\
& \leq& \exp\biggl\{(k+1) (\epsilon_dr)^2 - (
\epsilon_dr)^4\biggl(\frac{1}{4} + o(1)\biggr)
\biggr\}\int_{W_n \times B_{x_{i0}}(r)^k} \rho^{(k)}(\mathbf{x}) \,d
\mathbf{x}
\\
& \leq& n r^{2k} \exp\biggl\{k(\epsilon_dr)^2
- (\epsilon_dr)^4\biggl(\frac
{1}{4} + o(1)\biggr)
\biggr\}.
\end{eqnarray*}
It is easy to see that there exists a constant $C_{d,k} > 0$ such that
$\mathbb{E} \{C_k \} \to0$ for $r_n \geq C_{d,k}(\log
n)^{1/4}$, and so we are done.
\end{pf*}
%

\section{Morse theory for random geometric complexes}\label{secmorsergc}

Our aim in this section is to present a collection of results
concerning random geometric complexes, but from the viewpoint of Morse theory.

In fact, we have already used discrete Morse theory to derive some of
the connectivity thresholds for Vietoris--Rips complexes in Theorem~\ref{teoVRcontractgin}. However, in addition to this essentially
combinatorial Morse theory, there is a different and more geometric
version of Morse theory for nonsmooth functions on ``nice'' manifolds~\cite{gershkovichmorse1997}. While discrete Morse theory can be
applied to study simplicial complexes without requiring any information
on an ambient space in which the complex is embedded, in a geometric
setting such as ours one can exploit knowledge of the ambient
(Euclidean, in our case) space to apply the so-called ``min-type'' Morse theory.

This theory has also been exploited in the past to study of random
geometric complexes on Poisson and i.i.d. point processes in \cite
{Bobrowski11}, where it was shown that this Morse theoretic approach
can give an intrinsically richer set of results than that obtained by
attacking homology directly. Further, these Morse theoretic results
have, as usual, implications about Betti numbers. We wish to point out
that each of thse quite distinct versions of Morse theory have proved
to be useful tools in the study of random complexes.

We do not intend to give full proofs here, but rather to set things up
in such a way that parallels between the structures that have appeared
in previous sections and those that are natural to the Morse theoretic
approach become clear, and it becomes ``obvious'' what the Morse
theoretic results will be. Full proofs would require considerable more
space, but would add little in terms of insight. We note, however, that
this does not make the proofs of \cite{Bobrowski11} in any way
redundant. On the one hand, the results there go beyond what we have
here (albeit only for the Poisson and i.i.d. cases), and it is their
existence that allows us to be certain that the parallels work properly.

We start with some definitions and a quick description of the Morse
theoretic setting.

\subsection{Morse theory}\label{secmorsetheory}

Morse theory for geometric complexes is based on the distance function,
$d_{\Phi}\dvtx\mR^d \to\mR_+$, defined by
\[
d_{\Phi}(x):= \min_{X \in\Phi}\llVert x-X\rrVert,\qquad x \in
\mR^d.
\]
Note that while classical Morse theory deals with smooth functions, the
distance function is piecewise linear, but nondifferentiable along
subspaces. The extension to the distance function of classical Morse
theory is discussed in detail in \cite{Bobrowski11}, based on the
definitions and results in \cite{gershkovichmorse1997},
and we shall adopt the same approach. The main difference between
smooth Morse theory and that based on the distance function lies in the
definition of the indices of critical points.

Critical points of index $0$ of the distance function are the points
where $d_{\Phi} = 0$, which are local and global minima, and are the
points\vspace*{1pt} of $\Phi$. For higher indices, define the critical points as
follows: A point $c \in\mR^d$ is said to be a \textit{critical point}
with index $1 \leq k \leq d$ if there exists a collection of points $\X
= \{X_1,\ldots,X_{k+1}\} \subset\Phi^{(k+1)}$ such that the
following conditions hold:
\begin{longlist}[(2)]
\item[(1)] $d_{\Phi}(c) = \llVert c-X_i\rrVert$ for all $1 \leq i
\leq k+1$ and
$d_{\Phi}(c) < \llVert c-Y\rrVert$ for all $Y \in\Phi\setminus\X$.
\item[(2)] The points $X_i, 1 \leq i \leq k+1$ lie in general
position; namely, they do not lie in a $(k-1)$-dimensional affine space.
\item[(3)] $c \in \conv^o(\X)$, where $\conv^o(\X)$ denotes the
interior of the convex hull formed by the points of $\X$.
\end{longlist}
Let $C(\X)$ denote the center of the unique $(k-1)$-dimensional sphere
(if it exists) containing the points of $\X\in\Phi^{(k+1)}$ and
$R(\X)$ be the radius of the ball. The conditions in the definition of
critical points can be reduced to the following more workable
conditions; see \cite{Bobrowski11}, Lemma 2.2. A set of points $\X
\in\Phi^{(k+1)}$ in general position generates an index $k$ critical
point if and only if
\[
C(\X) \in \conv^o(\X)\quad \mbox{and}\quad \Phi\bigl(B_{C(\X)}\bigl(R(
\X)\bigr)\bigr) = 0.
\]
Our interest lies in critical points which are at most at a distance
$r$ from $\Phi$, namely, those for which $d_{\Phi}(c) \leq r$, or,
equivalently $R(\X) \leq r$. The reason for this lies in the simple
fact that
\begin{eqnarray*}
d_\Phi^{-1}\bigl([0,r]\bigr) = \bigcup
_{x\in\Phi} B_x(r),
\end{eqnarray*}
and, as we already noted earlier, by the nerve theorem this is homotopy
equivalent to the \v{C}ech complex $C(\Phi,r)$.

The following indicator functions will be required to draw the analogy
between counting critical points and counting components of random
geometric graphs. For $\X\in\Phi^{(k+1)}$, define
\begin{eqnarray*}
h(\X)&:=& \1\bigl[C(\X) \in \conv^o(\X)\bigr],
\\
h_r(\X)&:=&
\1\bigl[C(\X) \in \conv^o(\X)\bigr]\1\bigl[R(\X) \leq r\bigr].
\end{eqnarray*}
Note that these functions are translation and scale invariant, as were
the $h_{\Gamma}$ functions defined for the subgraph and component\vspace*{1pt}
counts in Section~\ref{secsubgraphrgg}; namely, for all $x \in\mR
^d$ and $\y= (0,y_1,\ldots,y_k) \in\mR^{d(k+1)}$,
\[
h_r(x,x + ry_1,\ldots,x + ry_k) =
h_1(\y).
\]
This was the key property of $h_{\Gamma}$ used to derive asymptotics
for component counts. Thus, once we manage to represent the numbers of
critical points as counting statistics of $h_r$, the analogy with
component counts is made.
To this end, let $N_k(\Phi,r)$ be the number of critical points of
index $k$ for the distance function $d_{\Phi}$ that are at most at a
distance $r$ from $\Phi$. Then
%
%
\begin{equation}
\label{eqncriticalpts} N_k(\Phi,r) = \sum_{\X\in\Phi^{(k+1)}}h_r(
\X)\1\bigl[\Phi\bigl(B_{C(\X
)}\bigl(R(\X)\bigr)\bigr) = 0\bigr].
\end{equation}
The similarity between the expression for $N_k(\Phi_n,r_n)$ and $J_n$
[cf. (\ref{eqnGn})] should convince the reader that the method of
proof used for component counts will also suffice for a derivation of
the asymptotics of Morse critical points. Although the void indicator
term is slightly different, we can use the fact that $R(\X) \leq r$
for $h_r(\X) = 1$ to apply the techniques of Section~\ref{secsubgraphrgg} with only minor changes.

\subsection{Limit theorems for expected numbers of critical points}\label{secexpmorsepts}
As in previous sections, we shall give results for the sparse and
thermodynamic regimes separately. In the Betti number results, in the
sparse regime ($r_n \to0$) the scaling factor of $n$ for $J_n$ (see
Theorem~\ref{propexpsubgrsp}) arose from the translation
invariance of $h_{\Gamma}$ and $\Phi$. The factor of $r_n^{d(k-1)}$
was due to the scale invariance of $h_{\Gamma}$, and the factor of
$f^k(r_n)$ came from the scaling of the joint intensities $\rho
^{(k)}$. Since $h_r$ is also translation and scale invariant, we work
under the same assumptions on $\Phi$ as in Theorems~\ref
{propexpsubgrsp} and~\ref{propcovbdssp} with corresponding
conditions $h_r$ in order to obtain asymptotics for expected number of
critical points of the distance function. Also, $\mathbb{E} \{
N_0(\Phi_n,r) \} = \mathbb{E} \{\Phi(W_n) \} = n$
for all $r \geq0$ and so we shall focus only on $N_k, 1 \leq k \leq
d$. The corresponding result is as follows:
%

\begin{theorem}[(Sparse regime)]\label{propmorseptssp}
Let $\Phi$ be a stationary point process in $\mR^d$ satisfying the
assumptions of Theorem~\ref{propexpsubgrsp} for all $1 \leq k \leq
(d+1)$. Let $r_n \to0$ and $\y= (0,y_1,y_2,\ldots,y_k)$. Then, for
$1 \leq k \leq d$,
\begin{eqnarray*}
\lim_{n \to\infty} \frac{\mathbb{E} \{N_k(\Phi_n,r_n)
\}}{n r_n^{dk} f^{k+1}(r_n)} & = & \nu_k(\Phi,0)
\\
&:= & \frac{1}{(k+1)!}\int_{\mR^{dk}}h_1(
\y)g^{k+1}_{\rho}(\y) \,d\y.
\end{eqnarray*}
Further, $\mathsf{Var} (N_k(\Phi_n,r_n) ) = O(\mathbb
{E} \{N_k(\Phi_n,r_n )\} )$ for negatively associated
point processes.
\end{theorem}

One point that is deserving of additional comment for the proof is
that, as in Theorem~\ref{propexpsubgrsp}, we can omit the void
probability term in the limit by the following reasoning: since $R(\y)
\leq r$ if $h_r(\y) = 1$, $\y= (0,y_1,\ldots,y_k)$, we have that
whenever $h_{r_n}(\y) = 1$,
\[
\bigl\{\Phi\bigl(B_{C(r_n^{1/d} \y)}\bigl(r_n^{1/d}\bigr)
\bigr) = 0 \bigr\} \subset\bigl\{\Phi\bigl(B_{C(r_n^{1/d} \y)}\bigl
(r_n^{1/d}R(
\y)\bigr)\bigr) = 0 \bigr\}, %
\]
and the probability of the left event here (and hence the right as
well) tends to $1$. This follows from similar arguments to those in the
proof of Theorem~\ref{propexpsubgrsp}.

Turning now to the thermodynamic regime, we saw in Theorem~\ref{propexpsubgrth}
that the sole scaling factor of $n$ for component
counts is due to the translation invariance of $h_{\Gamma}$ and $\Phi
$. The same remains true for mean numbers of critical points.
%

\begin{theorem}[(Thermodynamic regime: $r_n^d \to\beta$)]\label{propmorseptsth}
Let\vspace*{1pt} $\Phi$ be a stationary point process in $\mR^d$ satisfying the
assumptions of Theorem~\ref{propexpsubgrth} for all $1 \leq k \leq
(d+1)$. Let $r_n^d \to\beta\in(0,\infty)$ and $\y=
(0,y_1,y_2,\ldots,y_k)$. Then, for $1 \leq k \leq d$,
\begin{eqnarray*}
&& \lim_{n \to\infty} \frac{\mathbb{E} \{N_k(\Phi
_n,r_n )\} }{n}
\\
&&\qquad = \nu_k(\Phi,
\beta)
\\
&&\qquad:= \frac{\beta^k}{(k+1)!}\int_{\mR^{dk}}h_1(\y)
\P^{!}_{\beta^{1/d} \y}\bigl(\Phi\bigl(B_{C(\beta^{1/d} \y)}\bigl(
\beta^{1/d}R(\y)\bigr)\bigr) = 0\bigr)\rho^{(k)}\bigl(
\beta^{1/d} \y\bigr) \,d\y.
\end{eqnarray*}
Further, assume that $\Phi$ is also a negatively associated point
process such that
\[
\mathbb{P} \bigl\{ \Phi\bigl(B_{C(\mathbf{x})}\bigl(\beta^{1/d}\bigr)
\bigr) = 0 \bigr\} > 0
\]
for a.e. $\mathbf{x}= (0,x_1,\ldots,x_k) \in B_0(3\beta^{1/d})^{k+1}$,
and for all $1 \leq k \leq d$. Then $\nu_{k}(\Phi,\break \beta
) > 0$ for all $1 \leq k \leq d$.

Also,
$\mathsf{Var}
(N_k(\Phi_n,r_n) ) = O(\mathbb{E} \{N_k(\Phi
_n,r_n) \} )$ for negatively associated point processes.
\end{theorem}

As previously, the void probability needs some attention. In this case,
to show its positivity, we again use the fact that $R(\y) \leq1$ if
$h_1(\y) = 1$, and hence, whenever $h_1(\y) = 1$,
\[
\bigl\{\Phi\bigl(B_{C(\beta^{1/d} \y)}\bigl(\beta^{1/d}\bigr)\bigr) = 0
\bigr\} \subset\bigl\{\Phi\bigl(B_{C(\beta^{1/d} \y)}\bigl(\beta
^{1/d}R(\y)\bigr)\bigr)
= 0\bigr\}.
\]
The positivity of the first event under Palm probability is guaranteed
by our assumption via Lemma~\ref{lemNApalmvoid}.

Finally, we turn to a result about Euler characteristics that is not
accessible from the non-Morse theory. We
already defined the Euler characteristic in terms of Betti numbers at
(\ref{Eulerbybettiequn}), and showed in Theorem~\ref{teocechcontract} that, in the connectivity regime, it is~$1$ with
high probability. However, taking an alternative, but equivalent,
definition via numbers of Morse critical points, we can deduce its
$L_1$ asymptotics in the sparse and thermodynamic regimes as a
corollary of the previous results in this section. The alternative
definition, which is more amenable to computations due to the bounded
number of terms in the following sum, is
\[
\chi\bigl(C(\Phi,r)\bigr):= \sum_{k = 0}^d(-1)^kN_k(
\Phi,r).
\]
%

\begin{theorem}
\label{propEulercharspth}
Let $\Phi$ be a stationary point process in $\mR^d$ satisfying the
assumptions of Propositions~\ref{propexpsubgrsp} and~\ref
{propexpsubgrth} for all $1 \leq k \leq(d+1)$:
%
\begin{longlist}[(iii)]
\item[(i)] If $r_n \to0$, then
\[
n^{-1} \mathbb{E} \bigl\{{\chi\bigl(C(\Phi_n,r_n)
\bigr)} \bigr\} \to1.
\]
\item[(ii)] If $r_n^d \to\beta\in(0,\infty)$, then
\[
n^{-1}\mathbb{E} \bigl\{{\chi\bigl(C(\Phi_n,r_n)
\bigr)} \bigr\} \to1 + \sum_{k = 1}^d(-1)^k
\nu_k(\Phi,\beta).
\]

\item[(iii)]If $\Phi$ is also a negatively associated point process,
then the above convergences also hold in the $L_2$-norm.
\end{longlist}
\end{theorem}

To prove the part~(iii) of the theorem, we need variance bounds, which
is why we require the additional assumption of negative association.
For example, in the sparse regime, we have the following bound via the
Cauchy--Schwarz inequality:
\begin{eqnarray*}
&& \mathbb{E} \biggl\{\biggl\llVert\frac{\chi(C(\Phi_n,r_n))}{n}-1\biggr
\rrVert^2 \biggr\}
\\
&&\qquad  =  \mathbb{E} \Biggl\{\Biggl\llVert\biggl(
\frac{\Phi
(W_n)}{n}-1\biggr) + \sum_{k = 1}^d(-1)^k
\frac{N_k(\Phi,r)}{n}\Biggr\rrVert^2 \Biggr\}
\\
&&\qquad \leq d \Biggl(\frac{\mathsf{Var} (\Phi(W_n) )
}{n^2} + \sum_{k = 1}^d
\frac{\mathbb{E} \{N_k(\Phi
_n,r_n)^2 \} }{n^2} \Biggr)
\\
&&\qquad =  d \Biggl(\frac{\mathsf{Var} (\Phi(W_n) ) }{n^2} + \sum_{k = 1}^d
\frac{\mathsf{Var} (N_k(\Phi_n,r_n) )
}{n^2} + \sum_{k = 1}^d
\frac{(\mathbb{E} \{
N_k(\Phi_n,r_n) \} )^2}{n^2} \Biggr).
\end{eqnarray*}
The $L_2$ convergence follows once it is noted that all the terms on
right-hand side converge to $0$ due to the variance bounds proven for
negatively associated point processes. A slight modification of this
argument handles the thermodynamic regime as well.

\begin{appendix}\label{secAppendix}
\section*{Appendix}

In this section, we prove the result about Palm void probabilities of
Ginibre point process that is used in the proof of Theorem~\ref
{teoVRcontractgin}. The proof is due to Manjunath Krishnapur.
%

\begin{lem}
\label{lempalm-voidGinibre}
Let $D = B_0(r) \subset\mR^2$ for some $r > 0$ and $\Phi$ be the
Ginibre point process. Then for $k \geq1$ and $\mathbf{x}\in\mR^{2k}$,
\[
\P^!_{\mathbf{x}}\bigl\{\Phi(D) = 0\bigr\} \leq\exp\bigl\{kr^2
\bigr\}\mathbb{P} \bigl\{ \Phi(D) = 0 \bigr\} = \exp\bigl\{kr^2 -
r^4\bigl(\tfrac{1}{4} + o(1)\bigr) \bigr\}.
\]
\end{lem}

\begin{pf}
We shall prove the result for $k = 1$. The proof for the general case
then follows by a recursive application of the same argument.

Let $\cK_D$ be the restriction to $D$ of the integral operator $\cK$
corresponding to Ginibre point process. Since the Palm process of the
Ginibre point process is also a determinantal point process, let $\cL
_D$ be the integral operator corresponding to the Palm point process
restricted to $D$. Let $\lam_i, i =1,2,\ldots$ and $\mu_i,
i=1,2,\ldots$ be the eigenvalues of $\cK_D$ and $\cL_D$,
respectively. From \cite{Shirai03}, Theorem~6.5, we know that $\cK_D
- \cL_D$ has rank $1$, and hence, by a generalization of Cauchy's
interlacement theorem, the respective eigenvalues are interlaced with
$\lam_i \geq\mu_i \geq\lam_{i+1}$ for $i =1,2,\ldots.$

Now, consider the case $k=1$ and we have the following inequality:
\[
\P^!_{\mathbf{x}}\bigl\{\Phi(D) = 0\bigr\} = \prod
_{i \geq1} (1 - \mu_i) \leq\prod
_{i \geq2} (1 - \lam_i) = (1 - \lam_1)^{-1}
\mathbb{P} \bigl\{ \Phi(D) = 0 \bigr\},
\]
where the two equalities are due to \cite{Ben09} Theorem~4.5.3, and
the inequality is due to the generalization of Cauchy's interlacement
theorem described above. Now using \cite{Ben09},
Proposition~7.2.1, to bound $\mathbb{P} \{ \Phi(D) = 0 \}$ and
from the fact that $1 - \lam_1 = \mathbb{P} \{ \operatorname{EXP}(1) > r^2
\} = \exp\{-r^2\}$ (see \cite{Ben09}, proof of Theorem 4.7.1),
where $\operatorname{EXP}(1)$ is the exponential random variable with mean $1$, we
have the desired inequality for the case $k=1$.
\end{pf}
\end{appendix}

\section*{Acknowledgments}

The authors wish to thank Manjunath Krishnapur for pointing out an
error in Remark \ref{re3.2}(9) and other useful discussions regarding
determinantal point processes. The authors also thank an anonymous
referee for many helpful comments on the first draft. Most of this work
was done when D.~Yogeshwaran was a postdoctoral researcher at Technion, Israel and
D.~Yogeshwaran thanks Technion for its support.



%

\printaddresses
\end{document}